\documentclass[11pt, a4paper]{amsart}
\usepackage{amsmath,amssymb,amsfonts, float}
\usepackage{thmtools}

\usepackage[all]{xy}
\usepackage{caption}
\usepackage{enumerate}
\usepackage{mathpazo}
\usepackage{a4wide}
\usepackage{diagbox}

\usepackage{bm}
\usepackage{graphicx}
\usepackage[breaklinks=true]{hyperref}

\usepackage{xcolor}
\usepackage{multicol}
\usepackage{tabularx}

\usepackage{hyperref}
\usepackage[capitalize]{cleveref}

\usepackage[normalem]{ulem}

\newtheorem{thm}{Theorem}[section] 
\newtheorem*{thm*}{Theorem} 
\newtheorem{prop}[thm]{Proposition}
\newtheorem{lem}[thm]{Lemma}
\newtheorem{cor}[thm]{Corollary}

\theoremstyle{definition}
\newtheorem{definition}[thm]{Definition}

\newtheorem{question}[thm]{Question}
\newtheorem{rem}[thm]{Remark}

\DeclareMathOperator{\C}{\mathbb{C}}
\DeclareMathOperator{\Z}{\mathbb{Z}}

\DeclareMathOperator{\F}{\mathbb{F}}

\DeclareMathOperator{\Gal}{\text{Gal}}

    \DeclareFontFamily{U}{wncy}{}
    \DeclareFontShape{U}{wncy}{m}{n}{<->wncyr10}{}
    \DeclareSymbolFont{mcy}{U}{wncy}{m}{n}
    \DeclareMathSymbol{\Sha}{\mathord}{mcy}{"58}

\numberwithin{equation}{section}

\DeclareSymbolFont{bbold}{U}{bbold}{m}{n}
\DeclareSymbolFontAlphabet{\mathbbold}{bbold}

\usepackage{hyperref}

\newcommand{\Rad}{{\rm Rad}}
\newcommand{\s}{\text{ss}}

\begin{document}
\sloppy

\title{Sums of units in finite rings \\ and applications to Cayley graphs }
 \author{ J\'an Min\'a\v{c}, Tung T. Nguyen, Nguy$\tilde{\text{\^{e}}}$n Duy T\^{a}n }
 \dedicatory{Dedicated to Professor T. Tamizh Chelvam}
\address{Department of Mathematics, Western University, London, Ontario, Canada N6A 5B7}
\email{minac@uwo.ca}
\date{\today}

 \address{Department of Mathematics, Elmhurst University, Elmhurst, Illinois, USA}
 \email{tung.nguyen@elmhurst.edu}
 
  \address{
Faculty of Mathematics and 	Informatics, Hanoi University of Science and Technology, 1 Dai Co Viet Road, Hanoi, Vietnam } 
\email{tan.nguyenduy@hust.edu.vn}

\thanks{JM is partially supported by the Natural Sciences and Engineering Research Council of Canada (NSERC) grant R0370A01. TTN is partially supported by an AMS-Simons Travel Grant. NDT is partially supported by the Vietnam National
Foundation for Science and Technology Development (NAFOSTED) under grant number 101.04-2023.21}

\begin{abstract}
The question of whether a ring is additively generated by its units has been studied from several perspectives in ring theory and algebraic graph theory. In this paper, we investigate this problem for finite rings, not necessarily commutative, and relate it to
the connectedness of gcd-graphs, the existence of perfect state transfer, and the
solvability of certain equations over finite fields. Additionally, we discuss a generalization of this question in which only certain normalized units are allowed in the generating set. Our work intersects algebra, number theory, and graph theory, and may be of interest to a broad audience.

\end{abstract}

\maketitle
\tableofcontents 
\section{Introduction}
The question of which rings are additively generated by their units has a long and rich history, attracting mathematicians working in algebra, number theory, and graph theory. Inspired by the work of Zelinski in \cite{zelinsky1954every} on the endomorphism ring of vector spaces, Skornyakov asked in \cite[p.167, Problem 31]{Skornyakov1964} for a classification of rings that are additively generated by units. (The main theorem of Zelinski in \cite{zelinsky1954every}, which states that every matrix is a sum of two invertible matrices, was later reproved in \cite{invertible_matrices} by Lord, who was not aware of Zelinski's work.) More precisely, Skornyakov asked whether every element of a von Neumann regular ring that does not have $\F_2$ as a quotient is a sum of units. Shortly after, Bergman gave a counterexample to this question, which led to a series of works on this topic (see \cite{Ehrlich1968, Raphael1974}). Some further refinements of Skornyakov's question are investigated in \cite{goldsmith1998unit, maimani2010rings}. While \cite{goldsmith1998unit} takes a purely algebraic approach, \cite{maimani2010rings} looks at it through the lens of graph theory. More precisely, the authors of \cite{maimani2010rings} realize that a finite commutative ring $R$ is additively generated by units if and only if the associated unitary Cayley graph $G_{R}$ is connected. We recall that $G_{R}$ is the graph whose vertex set is $R$ and two vertices $a,b$ are adjacent if and only if $a-b$ is a unit. In the same spirit, in \cite{podesta2021_finitefield, podesta2025GP, podesta_new} the Waring problem on whether a ring is additively generated by the set of $k$-th powers of units has also been approached via the so-called generalized Paley graphs.

Our interest in this question stems from previous work on gcd-graphs (see \cite{klotz2007some, nguyen2024certain, nguyengcd2026}), $U$-unitary Cayley graphs (see \cite{nguyen2025supercharacters}), and the existence of perfect state transfer (PST) on them (see \cite{bavsic2009perfect, nguyen2025perfect, pst-gcd-new}). In these works, the sums of units play a crucial role in understanding the connectedness of these graphs, the explicit calculation of certain arithmetical sums, as well as whether PST can exist between two nodes. In particular, we affirmatively answer Skornyakov's question when $R$ is a finite ring, which is not necessarily commutative (the case where the ring is commutative is settled in \cite{maimani2010rings}).

In this article, we study some questions in the spirit of Skornyakov's original problem, but using other subsets $S$ of $R$. Typically, such an $S$ has an arithmetic origin. Here, the relevant graph that connects algebra and graph theory is the Cayley graph $\Gamma(R,S)$ whose vertex set is $R$ and two vertices $a,b$ are adjacent if $a-b \in S.$ As observed in \cite{maimani2010rings}, this connection comes from the fact that $R$ is additively generated by $S$ if and only if $\Gamma(R,S)$ is connected. In the special case where $S$ is stable under the left and right actions of $R^{\times}$, the Cayley graph $\Gamma(R,S)$ is an example of a gcd-graph, which is widely studied in the literature (see \cite{bavsic2009perfect, klotz2007some, nguyengcd2026, saxena2007parameters, pst-gcd-new}). In this particular case, our result presented here completely resolves the question of when $S$ additively generates $R$ using the geometry of $\Gamma(R,S)$. We also discuss some relationship between gcd-graphs and the total graphs introduced by Anderson and Badawi in \cite{anderson2008total}.

\subsection{Outline.}
In \cref{sec:sum_units}, we revisit the question of when a ring has the property that every element is a sum of two units. Here, we describe several interpretations of this property. We also study the inheritance of the $2$-sum property under ring homomorphisms and ring extensions. Furthermore, we consider a natural generalization of Skornyakov's original problem; namely, we classify all elements in a finite ring that can be expressed as the sum of two units. We apply this result to compute the precise value of the sum of all units in a finite ring. Additionally, we use it to give a complete classification of connected gcd-graphs. In \cref{sec:normalized_units}, we introduce a generalization of Skornyakov's problem in which only certain normalized units are allowed in the generating set. We focus our discussion on three standard classes of rings: matrix rings, group rings, and finite fields. In all these cases, there is a natural notion of normalized units. While the study of these rings is encompassed by the framework we have laid out, the proofs in each case are quite different in nature. Moreover, in the case of finite fields, there are edge cases that require an entirely different approach. It is quite remarkable that there is a \textit{complete} answer in this case for the $n$-sum property. In fact we provide a sharp answer by finding the minimal value of $n$ for each finite nontrivial extension of finite fields.
Finally, in \cref{sec:pst}, we use sums of units to study the \textit{non-existence} of perfect state transfer (PST) on Cayley graphs. In particular, we show that there is no PST on gcd-graphs defined over a ring that has the $n$-sum property.

\section{Sums of units} \label{sec:sum_units}
Let $R$ be a finite unital ring such that $0 \neq 1.$ In this article, we follow the convention that a ring homomorphism $\Phi: R \to R'$ would preserve $1$; namely, $\Phi(1)=1.$ We also say that $R'$ is a quotient of $R$ if there is a surjective ring homomorphism $\Phi: R \to R'.$ We will write $\F_q$ for the finite field with $q$ elements. 

\subsection{Equivalent conditions for rings with the $2$-sum property}

Let  $n$ be a positive integer.  In \cite{goldsmith1998unit}, the authors define the notion of a ring with the $n$-sum property. Namely, a ring $R$ is said to have the $n$-sum property if and only if every element in $R$ can be written as the sum of exactly $n$ elements in $R^{\times}.$  In \cite{unitary, goldsmith1998unit, maimani2010rings, nguyengcd2026}, the authors show that for a ring $R$, if $\F_2$ is not a quotient of $R$, then $R$ has the $2$-sum property.  Furthermore, using a graph-theoretic argument, \cite{maimani2010rings} (for finite commutative rings) and \cite{nguyen2026supercharacters} (for any finite ring) show that $R$ is additively generated by its units if and only if $R$ does not have a quotient of the form $\F_2 \times \F_2.$ We discuss below some reformulations and applications of these results. The first treats several equivalent conditions for $R$ to have the 2-sum property, covering its unit group, quotient rings, certain special units, and the unitary Cayley graph defined on $R$. To do so, we recall that in \cite{nagell1970type}, Nagel introduces the notion of an exceptional unit; namely, a unit $u$ is called exceptional if $1-u$ is also a unit.  We are now ready to state our theorem. 
\begin{thm} \label{thm:main-conditions}
The following conditions are equivalent 
\begin{enumerate}
\item $R$ has no quotient of the form $\F_2.$
    \item $R$ has the $2$-sum property. 
    \item $R$ has at least one exceptional unit. 
    \item The unitary Cayley graph $G_{R}$ is not bipartite. 
\end{enumerate}
\end{thm}
\begin{proof}
We will show that $(1) \implies (2) \implies (3) \implies (4) \implies (1).$

    Let us first prove that $(1) \implies (2).$ Let $R^{\s} =R/\Rad(R)$ be the semisimplification of $R$ with $\Rad(R)$ being the Jacobson radical of $R.$ Since an element in $R$ is a unit if and only if its image in $R^{\s}$ is a unit, $R$ has the $2$-sum property if and only if $R^{\s}$ does (see \cite[Proposition 3.40]{chudnovsky2024prime} and \cite[Proposition 1.3]{goldsmith1998unit}). By the Artin-Wedderburn theorem, $R^{\s}$ is isomorphic to a direct product of the form $\prod_{i=1}^r M_{n_i}(F_i)$ where $n_i \geq 1$ and $F_i$ is a finite field. By our assumption, none of these factors is $\F_2.$ By \cite{invertible_matrices}, each of these factors has the $2$-sum property. Consequently, $R^{\s}$ and hence $R$ has the $2$-sum property.

Let us now prove $(2) \implies (3).$ Since $R$ has the $2$-sum property, we can write $1=u+v$ where $u, v \in R^{\times}.$ By definition, $u$ is an exceptional unit in $R.$

Next, let us show that $(3) \implies (4).$ Suppose that $u$ is an exceptional unit. Then $\{1, u, 1+u \}$ forms a triangle in $G_{R}.$ Therefore, $G_{R}$ cannot be bipartite. 

Finally, let us show that $(4) \implies (1).$ Suppose to the contrary that $R$ has a quotient of the form $\F_2.$ Let $\Phi\colon R \to \F_2$ be such a quotient map. Let $U = \ker(\Phi)$ and $V= 1 +\ker(\Phi).$ Then if $x, y$ both belong to either $U$ or $V$ then $\Phi(x-y)=0.$ As a result, $x-y \not \in R^{\times}.$ By definition, $U, V$ form a partition of $G_{R}$ with the bipartite property. We conclude that $\Phi$ cannot exist.

\end{proof}

\begin{cor} \label{cor:2-property-quotient}
    Suppose that there is a ring map $\Phi\colon R \to  R'$. If $R$ has the $2$-sum property, then so does $R'.$ 
\end{cor}
\begin{proof}
Since $R$ has the $2$-sum property, $R$ has an exceptional unit $u.$ Then $\Phi(u)$ is an exceptional unit in $R'.$ By \cref{thm:main-conditions}, $R'$ has the $2$-sum property as well. 
\end{proof}
\begin{rem}
    The converse of \cref{cor:2-property-quotient} is not true. Here is a counterexample: let $R'$ be a ring with the $2$-sum property,  $R=\F_2 \times R'$ and $\Phi\colon  R \to R'$ is the canonical projection map.

    Even under the assumption that $R$ is a subring of $R'$, the converse is also not true. For example, let $R=\F_2 \times \F_2$ and $R'=M_2(\F_2).$ Let $\Phi$ be the diagonal embedding of $\F_2 \times \F_2$ into $M_2(\F_2).$ In this case $R'$ has the $2$-sum property but $R$ does not. Below, we discuss various situations where the converse of \cref{cor:2-property-quotient} holds. 
\end{rem}
First, we deal with the study of rings that are additively generated by units. In general, it is not true that under the same assumption in \cref{cor:2-property-quotient}, $R'$ is additively generated by units when $R$ is. For example, let $R=\F_2$, $R'=\F_2 \times \F_2$, and $\Phi\colon R \to R'$ be the diagonal embedding; namely $\Phi(a)=(a,a).$ Then $R$ is generated by units but $R'$ is not. However, this statement would be true if $R'$ is a quotient of $R.$

\begin{prop} \label{prop:quotient-generated}
        Suppose that $R'$ is a quotient of $R$. If $R$ is additively generated by units, then $R'$ is also additively generated by units. 

\end{prop}

\begin{proof}
    We give two proofs for this statement. The first one is purely algebraic. Suppose that $r \in R'.$ Let $\widehat{r}$ be a lift of $r$ to $R.$ Then, $\widehat{r}$ is a sum of units in $R.$ By projecting this property to $R'$ and observing this projection sends units to units, we conclude that $r$ is also a sum of units in $R'.$

The second proof is also along the same line, but it is somewhat more \textit{geometric}. The quotient map between $R$ and $R'$ induces a graph morphism between $G_{R}$ and $G_{R'}$ which is surjective at the vertex level. Since $G_{R}$ is connected, $G_{R'}$ is connected as well. Consequently, $R'$ is additively generated by units. 
\end{proof}

Next, we discuss the converse of \cref{cor:2-property-quotient} and \cref{prop:quotient-generated} for certain classes of ring extensions. First, we study the case of group rings. 

\begin{prop}
    Let $G$ be a finite group and $S$ a finite ring. Then, the following statements hold 
    \begin{enumerate}
        \item $S[G]$ is additively generated by units if and only if $S$ is additively generated by units. 
        \item $S$ has the $2$-sum property if and only if $S[G]$ has the $2$-sum property. 
    \end{enumerate}
\end{prop}
\begin{proof}
Let us first focus on $(1).$ Suppose that $S[G]$ is generated by units. Since the augmentation map $ \epsilon\colon S[G] \to S$ is surjective, \cref{prop:quotient-generated} implies that $S$ is also additively generated by units. Conversely, suppose that $S$ is additively generated by units. We claim that $S[G]$ is also additively generated by units. In fact, let us consider a basic element of $S[G]$ of the form $sg$ where $s \in S$ and $g \in G.$ Because $S$ is additively generated by units,  we can write $s= \sum_{j} u_j$ with $u_j \in S^{\times}.$ Therefore 
\[ sg = \sum_{j} u_jg. \]
By definition $u_jg \in S[G]^{\times}.$ Therefore, $sg$ belongs to the abelian group generated by elements of $S[G]^{\times}.$ Since each element of $S[G]$ is a sum of basic elements of the form $sg$, we conclude that $S[G]$ is additively generated by units.

Let us now focus on $(2).$ Suppose that $S$ has the $2$-sum property. Since $S$ is a subring of $S[G]$, \cref{cor:2-property-quotient} implies that $S[G]$ also has the $2$-sum property. Conversely, if $S[G]$ has the $2$-sum property then using the augmentation map $\epsilon\colon S[G] \to S$ and \cref{cor:2-property-quotient}, we conclude that $S$ also has the $2$-sum property. 
    \end{proof}

When $S$ is a finite field, we have the following simple yet interesting result. We remark that the third property follows from our graph-theoretic argument.

\begin{cor} \label{cor:F[G]}
    Let $F$ be a finite field and $G$ a finite group. 
    \begin{enumerate}
        \item  $F[G]$ is additively generated by units. 
        \item $F[G]$ has the 2-sum property if and only if $F \neq \F_2.$
        \item When $F=\F_2$, $\F_2[G]$ has no quotient of the form $\F_2 \times \F_2.$
    \end{enumerate}
\end{cor}
\begin{rem}
Here we provide a purely algebraic proof of the third part of \cref{cor:F[G]}. Let
$\Theta : \mathbb{F}_2[G] \longrightarrow \mathbb{F}_2$
be any surjective homomorphism. Then $\Theta(g)=1$ for all $g \in G$. Hence $\Theta=\epsilon$ is the augmentation homomorphism.

Now consider any surjective homomorphism
\[
\psi : \mathbb{F}_2[G] \longrightarrow \mathbb{F}_2 \times \mathbb{F}_2.
\]
Then
\[
\psi(\alpha) = \bigl(\psi_1(\alpha),\psi_2(\alpha)\bigr),
\qquad \text{for all } \alpha \in \mathbb{F}_2[G],
\]
where
\[
\psi_i : \mathbb{F}_2[G] \longrightarrow \mathbb{F}_2,
\qquad i=1,2,
\]
are also homomorphisms. As we observed above, $\psi_1=\epsilon=\psi_2.$ Hence
\[
\psi(\alpha)=\bigl(\epsilon(\alpha),\epsilon(\alpha)\bigr),
\qquad \forall \alpha \in \mathbb{F}_2[G].
\]

Thus
\[
\psi\bigl(\mathbb{F}_2[G]\bigr)
=
\mathbb{F}_2
\subsetneq
\mathbb{F}_2 \times \mathbb{F}_2,
\]
and $\psi$ is not surjective---a contradiction.
\end{rem}

Another class of rings where the converse of \cref{cor:2-property-quotient} holds is the following.
\begin{prop}
    Let $\Phi\colon R \to R'$ be an infinitesimal extension; namely, $\Phi$ is surjective, and the kernel of $\Phi$ is a nilpotent ideal. If $R'$ has the $2$-sum property, then so does $R.$
\end{prop}

\begin{proof}
We first claim that if $x \in (R')^{\times}$ and $\widehat{x}$ is a lift of $x$ to $R$ then $\widehat{x}$ is a unit in $R.$ In fact, let $y$ be the inverse of $x$ and $\widehat{y}$ is a lift of $y$ to $R$. Then $\Phi(\widehat{x} \widehat{y})=xy=1=\Phi(1).$ This implies that $(\widehat{x} \widehat{y} -1) \in \ker(\Phi)$. Consequently, $\widehat{x} \widehat{y} = 1 + m$ where $m \in \ker(\Phi)$. Because $\ker(\Phi)$ is nilpotent, $1+m \in R^{\times}$ and hence $\widehat{x} \in R^{\times}$. We remark that since $R$ is finite, right-invertible implies invertible (see \cite[Proposition 2.1]{perfect-unitary}.)

    Because $R'$ has the $2$-sum property, \cref{thm:main-conditions} implies that it has an exceptional unit $u$. Let $\widehat{u}$ be a lift of $u$ to $R.$ Since $\ker(\Phi)$ is nilpotent, both $\widehat{u}, 1-\widehat{u}$ are units in $R.$ By definition, $\widehat{u}$ is an exceptional unit in $R.$ Therefore, $R$ has the $2$-sum property. 
\end{proof}

\subsection{Elements in a finite ring which are the sum of two units}
In this section, we study the following question: which elements in a finite ring $R$ can be written as a sum of two units? Let us first start with a motivational problem; namely, we study the sum of all units in a finite ring $R$. The case where $R$ has the $2$-sum property is relatively easy. 

\begin{prop} \label{prop:sum-units}
    Suppose that $R$ has the $2$-sum property. Let $S$ be a set which is stable under the left (or right) action of $R^{\times}$; namely, $R^{\times} S \subset S.$ Then 
    \[ T_{S} = \sum_{s \in S} s = 0.\]
    In particular 
    \[ T_{R^{\times}}=\sum_{u \in R^{\times}}u = 0.\]
\end{prop}
\begin{proof}

We give the proof for the case where \(S\) is stable under the left action; the right-action case is analogous.
    Since $S$ is left stable under the action of $R^{\times}$, 
    \[ uT_{S} = \sum_{s \in S} us = \sum_{s \in S} s= T_{S}. \] 
   Therefore, $(u-1) T_{S}=0$ for all $u \in R^{\times}.$ Furthermore, since $R$ has the $2$-sum property, by \cref{thm:main-conditions}, it has an exceptional unit $u.$ Since $u-1 \in R^{\times}$, this would imply that $T_{S}=0.$
\end{proof}

One may wonder what happens in \cref{prop:sum-units} if $R$ does not have the $2$-sum property. The argument given above shows that for all $u, v \in R^{\times}$, $(u+v)T_{S}=0$.   We see that, in order to answer this question,  we need to dive a bit deeper into Skornyakov's problem. More precisely, we will need to answer the question raised at the beginning of this section: which elements of a ring $R$ can be written as the sum of two units? In \cite{nguyengcd2026}, we study this problem when the ring $R$ is finite and commutative. We will show below that the same argument works for any finite ring as well. Let us recall that by the Artin-Wedderburn theorem, $R^{\s}$ is isomorphic to a product of matrix rings $M_{n_i}(F_i)$ where $n_i \geq 1$ and $F_i$ is a finite field. Let $r$ be the number of factors of $R^{\s}$ such that $n_i=1$ and $\F_i=\F_2.$ Then  $R^{\s} = \F_2^r \times R_0$ 
where $R_0$ does not have $\F_2$ as a quotient. In particular, $R_0$ has the $2$-sum property by \cref{thm:main-conditions}. We are now ready to state our result.

\begin{prop} \label{prop:which-element-sum-2}
  Let $\Phi$ be the composition $R \to R^{\s}:=\F_2^r \times R_0 \to \F_2^r$ where $r$ is as above. Let $\ker(\Phi)$ be its kernel. Then, every element of $\ker(\Phi)$ is the sum of two units. Conversely, an element of $R$ can be written as the sum of two units if and only if it is in $\ker(\Phi).$
\end{prop}

\begin{proof}
Suppose that $x \in \ker(\Phi)$. Let $\bar{x}$ be the image of $x$ in $R^{\s}=\F_2^r \times R_0.$ Because $\bar{x} \in \ker(\Phi)$, $\bar{x}$ is of the form $(0,s)$ where $s \in R_0.$ We know that $R_0$ has the $2$-sum property, so we can write $s=u_1+u_2$ where $u_1, u_2 \in R_0^{\times}.$ We then see that $\bar{x} $ has the following presentation as the sum of two units
\[ \bar{x}=(1,u_1)+(1,u_2).\]
Let $v_1$ (respectively $v_2$) be the lift of $(1,u_1)$ (respectively $(1,u_2)$) to $R.$ Then $v_1, v_2 \in R^{\times}$ and $x =v_1+v_2+m$ for some $m \in \Rad(R).$ Since $m \in \Rad(R)$, $v_2+m \in R^{\times}.$ By letting $v_3=v_2+m$, we see that $x$ is the sum of two units; namely $v_1$ and $v_3.$

Conversely, suppose that $x$ is the sum of two units. Then $\Phi(x)$ is also the sum of two units. Since the only unit in $\F_2^r$ is $1$, we conclude that $\Phi(x)=1+1=0$. Equivalently, $x \in \ker(\Phi).$
\end{proof}
Let us keep the same notation as in \cref{prop:which-element-sum-2}. We then have the following corollary. 
\begin{cor}
    Suppose that $S$ is stable under the action of $R^{\times}.$ Let 
    \[ T_{S} = \sum_{s \in S} s.\]
    Then $aT_{S}=0$ for all $a \in \ker(\Phi).$

\end{cor}

    \begin{proof}
The fact that $aT_{S}=0$ follows from the same argument as in \cref{prop:sum-units} and the fact that every element in $\ker(\Phi)$ is the difference of two units.  In fact, we can find $u_1, u_2 \in R^{\times}$ such that $a=u_1-u_2.$ Since $u_1 T_{S}=u_2T_S=T_S$, we conclude that $aT_{S}=0.$
\end{proof}
When $S=R^{\times}$, we can calculate $T_{R^{\times}}$ explicitly. 
\begin{prop}
 $T_{R^{\times}}=\sum_{u \in R^{\times}}u \neq 0$ if and only if $R$ is one of the following rings. 
    \begin{enumerate}
        \item $2=0$ in $R$, $|\Rad(R)|=2$ and $R/\Rad(R) \cong \prod_{i=1}^d F_i$ where each $F_i$ is a finite field of characteristic $2$. In this case, $T_{R^{\times}}=e$ where $e$ is the unique non-zero element in $\Rad(R).$
        \item $\Rad(R)=0$ and $R \cong \F_2^r \times \left(\prod_{i=1}^t F_i \right)$ where $r \geq 1$, $F_i$ is a finite field of characteristic $2$ and $|F_i|>2$ for each $1 \leq i \leq t. $ In this case $T_{R^{\times}}=(1_{\F_2^r}, 0).$
    \end{enumerate}
\end{prop}
\begin{proof}
First, we observe that if $R=R_1 \times R_2$ then 
\[ T_{R^{\times}} = (|R_2^{\times}| T_{R_{1}^\times}, |R_1^{\times}| T_{R_2^{\times}}).\]

In fact 
\begin{align*}
T_{R^{\times}} &= \sum_{u_1 \in R_1^{\times}, u_2 \in R_2^{\times}} (u_1, u_2) = \sum_{u_1 \in R_1^{\times}} \sum_{u_2 \in R_2^{\times}}(u_1, u_2)\\ 
&= \sum_{u_1 \in R_1^{\times}}(|R_2^{\times}| u_1, T_{R_2^{\times}}) =(|R_2^{\times}| T_{R_{1}^\times}, |R_1^{\times}| T_{R_2^{\times}}). 
\end{align*}
We consider three cases. \\
\textbf{Case 1.} $2 \neq 0$ in $R.$ Then, we can decompose the set of units into pairs $\{u, -u\}.$ In this case, $T_{R^{\times}}=0.$  \\
\textbf{Case 2.} $2=0$ in $R$ and $\Rad(R) \neq 0.$ 

Let $e$ be an arbitrary non-zero element in  $\Rad(R).$ Then, we can decompose $R^{\times}$ into pairs $\{u, u+e\}.$ We then see that 
\[ T_{R^{\times}} = \frac{|R^{\times}|}{2}e.\]
We remark that there is an exact sequence of (multiplicative)  groups
\[ 1 \to 1+ \Rad(R) \to R^{\times} \to (R/\Rad(R))^{\times} \to 1. \]
Furthermore, since $2=0$ in $R$, $1+\Rad(R)$ is a $2$-group. Therefore, we can conclude that $\dfrac{|R^{\times}|}{2} \neq 0$ in $R$ if and only if $\dfrac{|R^{\times}|}{2}$ is odd. By the above exact sequence, this happens if and only if $|\Rad(R)|=2$ and $|(R/\Rad(R))^{\times}|$ is odd. By the Artin-Wedderburn theorem, $R/Rad(R)=\prod_{i=1}^d M_{n_i}(F_i)$. We know that 
\[|M_{n_i}(F_i)^{\times}| = \prod_{j=0}^{n_i-1}(|F_i|^{n_i}-|F_i|^j), \]
which is even unless $n_i=1.$ We conclude that $R/\Rad(R)$ is a product of finite fields of characteristic $2.$ In this case, $e$ is the unique non-zero element in $\Rad(R)$ and $T_{R^{\times}}=\dfrac{|R^{\times}|}{2}e=e.$\\
\textbf{Case 3.} $2=0$ in $R$ and $\Rad(R)=0.$ In this case $R=\prod_{i=1}^d M_{n_i}(F_i)$ where $n_i \geq 1$ and each $F_i$ is a finite field of characteristic $2.$ We then have 
\[ T_{R^{\times}} = \left( \prod_{i \neq 1} |M_{n_i}(F_i)^{\times}|, \ldots, \prod_{i \neq d}|M_{n_i}(F_i)^{\times}| \right) (T_{M_{n_1}(F_1)^{\times}}, \ldots, T_{M_{n_d}(F_d)^{\times}}).\]
In other words, for each \(1\leq i\leq d\), the \(i\)-th component of \(T_{R^\times}\) is
\[
\left(\prod_{j\neq i}|M_{n_j}(F_j)^\times|\right)
T_{M_{n_i}(F_i)^\times}.
\]
We claim that in order for $T_{R^{\times}} \neq 0$, a necessary condition is $n_i=1$ for all $1 \leq i \leq d.$ Suppose to the contrary that it is not the case. Without loss of generality, assume that $n_1>1.$ Since $M_{n_1}(F_1)$ has the $2$-sum property , we conclude that the first component of \(T_{R^\times}\) is $0.$ Additionally, similar to the previous case, we have 
\[|M_{n_1}(F_1)^{\times}| = \prod_{j=0}^{n_1-1}(|F_1|^{n_1}-|F_1|^j), \]
which is even. Since $2=0$ in $R$, this would imply that for each $j \neq 1$, the $j$-th component of $T_{R^{\times}}$ is also $0.$ Consequently, $T_{R^{\times}}=0.$

Let us now assume that $n_i=1$ for all $1 \leq i \leq d.$ Then $|M_{n_i}(F_i)^{\times}|=|F_i^{\times}|=|F_i|-1$ which is odd. We then have 
\[ T_{R^{\times}} = (T_{M_{n_1}(F_1)^{\times}}, \ldots, T_{M_{n_d}(F_d)^{\times}}).\]
We know that $T_{F_i^{\times}}=0$ if $|F_i|>2$ and $T_{F_i^{\times}}=1$ if $F_i =\F_2.$ Therefore, $T_{R^{\times}} \neq 0$ if and only if $n_i=1$ for each $1 \leq i \leq d$ and there exists $1 \leq i \leq d$ such that $F_i=\F_2.$
\end{proof}
We now utilize \cref{prop:which-element-sum-2} to give a complete classification of gcd-graphs that are connected. We recall that a Cayley graph $\Gamma(R,S)$ is called a gcd-graph if $S$ is stable under the left and right actions of $R^{\times}$; namely, $R^{\times} S R^{\times}=S.$ This kind of graph is widely studied in the literature, starting with the work of Klotz-Sander on gcd-graphs over the cyclic ring $\Z/n$. Later developments study gcd-graphs over polynomial rings, unique factorization domains, finite chain rings (see \cite{kiani2011energy, minavc2024gcd, suntornpoch2016cayley}), and culminating with the most general definition in \cite{nguyengcd2026}.

Let \(S_1,\ldots,S_m\) be the orbits of the double coset \(R^\times\backslash R/R^\times\).  Then as explained in \cite{nguyen2025supercharacters}, $\Gamma(R,S)$ is a gcd-graph if $S$ is a disjoint union of some of the cosets $S_i$'s. Namely, there exists a subset $I \subset \{1, 2, \ldots, m\}$ such that $S= \bigsqcup_{i \in I} S_i.$

For each $x \in R$, we denote by $I_x$ to be the two sided ideal generated by $x.$ By definition, $I_x$ is the following set 
\[ I_x = \left\{ \sum_{i} a_i x b_i \mid a_i, b_i \in R \right\}.\]
We remark that if $x$ and $y$ belong to the same coset; namely $x=u yv$ where $u, v \in R^{\times}$ then  $I_x = I_y.$

\begin{thm}
Let \(\Phi:R\to \F_2^r\) be the composition
\[
R\to R^{ss}\cong \F_2^r\times R_0 \to \F_2^r,
\]
where \(R_0\) has no quotient isomorphic to \(\F_2\). Let $\Gamma(R,S)$ be a gcd-graph over $R$ where  $S = \bigsqcup_{i \in I} S_i$ for some $I \subset \{1, 2, \ldots, m \}$. Let $I_i$ be the two-sided ideals generated by $x_i \in S_i$ ($I_i$ is independent of the choice of $x_i$). Then $\Gamma(R,S)$ is connected if and only if the following conditions hold 
 \begin{enumerate}
     \item $\sum_{i \in I} I_i =R.$
     \item The cube-like graph $\Gamma(\F_2^r, \Phi(S))$ is connected. 
 \end{enumerate}
\end{thm}

\begin{proof}
    Suppose that $\Gamma(R,S)$ is connected. The map $\Phi$ induces a graph morphism $\Gamma(R,S) \to \Gamma(\F_2^r, \Phi(S))$. We note that, $\Phi(S)$ may contain $0$. In this case, $\Gamma(\F_2^r, \Phi(S))$ will have a simple loop at each vertex, which does not affect its connectivity. Since $\Gamma(R,S)$ is connected and $\Phi$ is surjective at the vertex level, $\Gamma(\F_2^r, \Phi(S))$ is connected as well.

    Let $H$ be the abelian group generated by $S.$ Then, $H$ is the set of elements of the form $\sum_{i} n_k s_k$, where $n_k \in \Z$ and $s_k \in S.$ We see that $H \subset \sum_{i\in I} I_i.$ Furthermore, since $\Gamma(R,S)$ is connected, $H=R$. This shows that $\sum_{i \in I} I_i =R$ as well.

Conversely, suppose that both of the above conditions are satisfied. Let $r \in R.$ We will show that $r \in H.$ In fact, since $\Gamma(\F_2^r, \Phi(S))$ is connected, we can write $\Phi(r) = \sum_{k} n_k \Phi(s_k),$ where $n_k \in \Z$ and $s_k \in S.$ By definition $(r-\sum_{k} n_ks_k) \in \ker(\Phi).$ Since $\sum_{k} n_k s_k \in H$, it is sufficient to show that $\ker(\Phi) \subset H.$ In fact, let $t \in \ker(\Phi)$. Because $\sum_{i \in I} I_i=R$, we can write  $1 = \sum_{i \in I} \sum_{k} a_{ik} x_i b_{ik},$ where $a_{ik}, b_{ik}  \in R, x_i \in S_i \subset S.$ Furthermore, as we explained previously, since $\Gamma(\F_2^r, \Phi(S))$ is connected, we can also write $1 = p+ \sum_{k} m_k x_k $ where $p \in \ker(\Phi)$, $x_k \in S$ and $m_k \in \Z.$ We then have  $t = \sum_{k} m_k (t x_k) + tp.$
 Since $t \in \ker(\Phi)$, $t$ is a sum of two units; say $t=u+v$ where $u,v \in R^{\times}.$ Then 
 $ tx_k = ux_k +v x_k.$  By definition $ux_k, vx_k \in S$ and hence $tx_k \in H.$ We claim that $tp \in H$ as well. In fact 
\[ tp = t(1)p = \sum_{i \in I} \sum_{k} (ta_{ik}) x_i (b_{ik}p). \]
Since \(\ker(\Phi)\) is a two-sided ideal, both \(ta_{ik}\) and \(b_{ik}p\) belong to \(\ker(\Phi)\). Therefore each of $ta_{ik}$ and $b_{ik}p$ is a sum of two units. Furthermore, since $R^{\times} S R^{\times}=S$, we conclude that $(ta_{ik}) x_i (b_{ik}p)$ is the sum of four elements in $S.$ Indeed, if
\[
ta_{ik}=u_1+u_2,\qquad b_{ik}p=v_1+v_2
\]
with \(u_1,u_2,v_1,v_2\in R^\times\), then
\[
(ta_{ik})x_i(b_{ik}p)
=
\sum_{\alpha,\beta=1}^2 u_\alpha x_i v_\beta,
\]
and each \(u_\alpha x_i v_\beta\in S\). This shows, in particular,  that $(ta_{ik}) x_i (b_{ik}p) \in H$ and therefore $tp \in H$. We conclude that $t \in H$ and hence $H=R.$ In other words, $\Gamma(R,S)$ is connected.  
\end{proof}

Finally, we conclude this section with some discussion of the unit graph (see \cite{MR1041619} for a survey on this topic). This type of graphs and their variants, such as total graphs, have been extensively studied in the literature (see, for example, \cite{anderson2008total, Ashrafi2010}). We recall that the unit graph $U_{R}$ is the graph whose vertex set is $R$. Additionally, two vertices $a,b$ are adjacent if $a+b \in R^{\times}.$ By a direct analogy with gcd-graphs, we can define the notion of a total gcd-graph as follows. 

\begin{definition}
    Let $S$ be a set which is stable under the left and right actions of $R^{\times}.$ The total gcd-graph of $R$ with respect to $S$, denoted by $T(R,S)$ is the graph whose vertex set is $R$. Furthermore, $a$ and $b$ are adjacent if and only if $a+b \in S.$
\end{definition}

While it is not true in general that $T(R,S)$ and $\Gamma(R,S)$ are isomorphic (see \cite{Huadong, shekarriz2012total} for some study on the relationship between these two graphs), the following statement holds. 

\begin{prop} 
The following conditions are equivalent. 
\begin{enumerate}
\item  $\Gamma(R,S)$ is connected. 
\item $S$ additively generates $R.$
\item $T(R,S)$ is connected. 
\end{enumerate}

\end{prop}

\begin{proof}
    The equivalence between $(1)$ and $(2)$ is well-known. The equivalence between $(2)$ and $(3)$ follows from an identical argument given in \cite[Theorem 4.1]{Alluqmani2026}. 
\end{proof}
\section{Sums of normalized units} \label{sec:normalized_units}
In this section, we study the $n$-sum property of $R$ with respect to a general subset of $R.$ In particular, we investigate the case where these subsets have arithmetic origins. To do so, we introduce the following formal definition.

\begin{definition}
    Let $S$ be a subset of $R.$ We say that $R$ has the $n$-sum property with respect to $S$ if every element of $R$ can be written as a sum of exactly $n$ elements in $S.$
\end{definition}

The case $S=R^{\times}$ is precisely  Skornyakov's question. On the other hand, if $S=(R^{\times})^k$ where $k$ is a positive integer, our definition is directly related to the Waring problem (see \cite{podesta2021_finitefield, podesta2025GP, podesta_new} for some work on this topic where the authors also approach this problem via graph theory). We remark that if $S \subset S'$ and $R$ has the $n$-sum property with respect to $S$, then $R$ also has the $n$-sum property with respect to $S'.$ In particular, if $S \subset R^{\times}$, then $R$ has the $2$-sum property with respect to $S$ only if $R$ has no quotient of the form $\F_2.$ For the rest of this article, we will denote by $U$ a subgroup of $R^{\times}.$ Since we only deal with undirected graphs, we will also assume that $-1 \in U.$

We first discuss a general result that studies the $n$-sum property with respect to a quotient map. 
\begin{prop} \label{prop:infinitesiamal-quotient}
    Let $R$ be a finite ring, and let $R'$ be a quotient of $R$ equipped with a quotient map $\Phi\colon R \to R'.$  Let $U$ be a subgroup of $(R')^{\times}$ and $U_{\Phi} =\{r \in R^{\times} \mid \Phi(r) \in U \}.$ Suppose that $\ker(\Phi) \subset \Rad(R)$. Then $R$ has the $n$-sum property with respect to $U_{\Phi}$ if and only if $R'$ has the $n$-sum property with respect to $U.$
\end{prop}

\begin{proof}
    Suppose that $R$ has the $n$-sum property with respect to $U_{\Phi}.$ Let $u \in R'$ and $\widehat{u}$ be a lift of $u$ to $R.$ Since $R$ has the $n$-sum property with respect to $U_{\Phi}$ we can write $\widehat{u} = \sum_{i=1}^n x_i$ where $x_i  \in U_{\Phi}.$ We then have $u = \Phi(\widehat{u})= \sum_{i=1}^n \Phi(x_i).$ 
    By definition, $\Phi(x_i) \in U.$ We conclude that $R'$ has the $n$-sum property with respect to $U.$

    Conversely, suppose that $R'$ has the $n$-sum property with respect to $U.$ We will show that $R$ also has the $n$-sum property with respect to $U_{\Phi}$.

    First, we claim that if $y \in (R')^{\times}$ and $\widehat{y}$ is a lift of $y$ then $\widehat{y} \in R^{\times}.$ In fact, let $t$ be the inverse of $y$; namely $yt=ty=1.$ Let $\widehat{t}$ be a lift of $t.$ Then $\widehat{y} \widehat{t}=1+a$ for some $a \in \ker(\Phi).$ By our assumption, $\ker(\Phi) \subset \Rad(R)$, so $1+a \in R^{\times}.$ This implies that $\widehat{y} \in R^{\times}$ as well. As a consequence of this observation, we conclude that $U_{\Phi}= \Phi^{-1}(U).$

    Now, let $v \in R$, then $\Phi(v)=\sum_{i=1}^n y_i$ for $y_i \in U.$ We then have $v =m + \sum_{i=1}^n \widehat{y_i} = (m+\widehat{y_1}) + \sum_{i=2}^n \widehat{y_i} $ where $\widehat{y_i}$ is a lift of $y_i$ in $R$ and $m \in \ker(\Phi)$. Since $U_{\Phi}=\Phi^{-1}(U)$, we know that $\widehat{y_1}+m  \in U_{\Phi} $ and $\widehat{y_i} \in U_{\Phi}$ for $2 \leq i \leq n.$  We conclude that $v$ can be written as the sum of $n$ elements in $U_{\Phi}$. We conclude that $R$ has the $n$-sum property with respect to $U_{\Phi}.$
\end{proof}
We discuss below some partial results for standard choices of $R$ and $U$ in the literature. 

\subsection{Sums of normalized units in matrix rings}
First, we deal with the case of matrix rings with coefficients in a finite field. 

\begin{prop}
Let $F$ be a field and $n \geq 2.$ Let $R=M_n(F)$ and $U= \pm SL_n(F)$ where $SL_n(F)=\{A \in M_n(F) \mid \det(A)=1\}.$ Then $R$ has the $2$-sum property with respect to $U.$
\end{prop}

\begin{proof}
Let \(C\in M_n(F)\). We claim there exist \(A,B\in SL_n(F)\) such that
\[
C=A-B.
\]

Let \(r=\operatorname{rank}(C)\). By the theory of Gauss elimination, there exist \(P,Q\in GL_n(F)\) with
\[
PCQ=\begin{bmatrix}I_r & 0\\[4pt]0 & 0\end{bmatrix}=I_r\oplus 0_{n-r}.
\]
Write \(C':=PCQ\). It suffices to write \(C'=A'-B'\) with \(\det(A')=\det(B')= \alpha\) where $\alpha = \det(P) \det(Q)$, since then
\[
C=P^{-1}A'Q^{-1}-P^{-1}B'Q^{-1}
\]
and \(P^{-1}A'Q^{-1},P^{-1}B'Q^{-1}\in SL_n(F)\).

Our approach is based on the following determinant identity for an \(n\times n\) matrix of the cyclic form
\[
\det\begin{bmatrix}
x_1 & a_1 &  &  &  \\
 & x_2 & a_2 &  &  \\
 &  & \ddots & \ddots & \\
 &  &  & x_{n-1} & a_{n-1} \\
a_n &  &  &  & x_n
\end{bmatrix}
= x_1x_2\cdots x_n + (-1)^{\,n+1}a_1a_2\cdots a_n.
\]
In this matrix, every other entry is $0.$ We construct matrices \(A'\) and \(B'\) of the same cyclic type so that \(A'-B'=I_r\oplus 0_{n-r}\). Concretely, choose entries \(x_i\) and \(a_i\) and set

\[
A'=
\begin{bmatrix}
1+x_1 & a_1 & 0 & \cdots & 0 & 0 & \cdots & 0\\
0 & 1+x_2 & a_2 & \cdots & 0 & 0 & \cdots & 0\\
\vdots & & \ddots & \ddots & \vdots & \vdots & & \vdots\\
0 & 0 & \cdots & 1+x_r & a_r & 0 & \cdots & 0\\
0 & 0 & \cdots & 0 & 0 & a_{r+1} & \cdots & 0\\
\vdots & \vdots & & \vdots & \vdots & \vdots & \ddots & \vdots\\
0 & 0 & \cdots & 0 & 0 & 0 & \cdots  & a_{n-1} \\
a_n & 0 & \cdots & 0 & 0 & 0 & \cdots  & 0 \\
\end{bmatrix},
\]

where the first $r$ diagonal entries are \(1+x_1,\dots,1+x_r\) and the remaining $(n-r)$ diagonal entries are 0. Similarly, define  

\[
B'=
\begin{bmatrix}
x_1 & a_1 & 0 & \cdots & 0 & 0 & \cdots & 0\\
0 & x_2 & a_2 & \cdots & 0 & 0 & \cdots & 0\\
\vdots & & \ddots & \ddots & \vdots & \vdots & & \vdots\\
0 & 0 & \cdots & x_r & a_r & 0 & \cdots & 0\\
0 & 0 & \cdots & 0 & 0 & a_{r+1} & \cdots & 0\\
\vdots & \vdots & & \vdots & \vdots & \vdots & \ddots & \vdots\\
0 & 0 & \cdots & 0 & 0 & 0 & \cdots  & a_{n-1} \\
a_n & 0 & \cdots & 0 & 0 & 0 & \cdots  & 0 \\
\end{bmatrix}.
\]
Then $A'-B' =I_r\oplus 0_{n-r}. $ We consider two cases. \\ 

\textbf{Case 1}: \(r=n\). Choose
\[
x_1=-1,\quad x_2=\cdots=x_n=0,\quad a_1=(-1)^{\,n+1}\alpha,\quad a_2=\cdots=a_n=1,
\]
Because $x_1+1=0$
\[
\det(A')= (1+x_1) \cdots(1+x_n)+\;(-1)^{\,n+1}a_1\cdots a_n=\alpha,
\]
Similarly $\det(B')=\alpha$. 

For example, when \(n=3\),
\[
I_3=\begin{bmatrix}0&\alpha&0\\[2pt]0&1&1\\[2pt]1&0&1\end{bmatrix}
-\begin{bmatrix}-1&\alpha&0\\[2pt]0&0&1\\[2pt]1&0&0\end{bmatrix}.
\]

Case 2: \(r<n\). Take \(x_{1}=x_2=\cdots=x_n=0\), and set
\[
a_1=(-1)^{\,n+1}\alpha,\qquad a_2=\cdots=a_n=1.
\]
Then 
\[
\det(A')=(1+x_1)\cdots(1+x_r)\cdot 0+\;(-1)^{\,n+1}a_1\cdots a_n=\alpha,
\]
\[
\det(B')=x_1\cdots x_r\cdot 0+\;(-1)^{\,n+1}a_1\cdots a_n=\alpha,
\]
For example, for $n=3$ and $r=2$, we have

\[
I_2 \oplus 0_1=\begin{bmatrix}1& \alpha& 0\\ 0&1&1\\ 1&0&0 \end{bmatrix}
-\begin{bmatrix}0&\alpha&0\\[2pt]0&0&1\\[2pt]1&0&0\end{bmatrix}.
\qedhere
\]
\end{proof}

We remark that we actually proved a stronger result: every matrix is a difference of two elements of $SL_n(F)$.
By the same argument as in \cref{prop:sum-units}, we have the following corollary. 

\begin{cor}
Let $n \geq 2$ and $F$ be a finite field. Then 
    $\sum_{u \in SL_n(F)} u =0. $
\end{cor}
\subsection{Sums of units in a group ring}
Let $F$ be a finite field and $U \subset F^{\times}$ such that $F$ has the $n$-sum property with respect to $U.$ This necessarily implies that $F \neq \F_2$ and $n \geq 2$. We will make this assumption throughout this section.

Let $G$ be a finite group and $R=F[G]$ be the group algebra of $G$ with coefficients in $F$. Let $\varepsilon_{G}\colon F[G] \to F$ be the augmentation map and $U_{\varepsilon_G}= \{r \in F[G]^{\times} \mid \varepsilon_G(r) \in U\}.$  In this section, we study the following question.

\begin{question} \label{question:F[G]}
  Suppose that $F$ has the $n$-sum property with respect to $U.$  Is it true that $F[G]$ also has the $n$-sum property with respect to $U_{\varepsilon_G}$? 
\end{question}
\begin{rem}
We can easily show that $F[G]$ always has the $N$-sum property with respect to $U_{\varepsilon_G}$ where $N= n|G|$.  In fact, a basic element of $F[G]$ is of the form $\sum_{g} a_g g.$ We can write each $a_g$ as the sum of $n$ terms in $U.$ Each of these terms when multiplied by $g$ is an element of $U_{\varepsilon_G}. $ \cref{question:F[G]} asks for a stronger statement that $F[G]$ has the $n$-sum property. 
\end{rem}

First, we deal with the easy case where $|G|$ is a $p$-group with $p =\text{char}(F)$. In this case the answer to \cref{question:F[G]} is affirmative. 
\begin{prop} \label{prop:p-group}
Suppose that $G$ is a $p$-group with $p = \text{char}(F).$ Then $F[G]$ has the $n$-sum property  with respect to $U_{\varepsilon_G}$. 
\end{prop}
\begin{proof}
    It is well-known that if $G$ is a $p$-group and $p=\text{char}(F)$ then $\Rad(F[G]) = \Delta(G)$ where $\Delta(G)=\ker(\varepsilon_{G}).$ Since $F[G]/\Delta(G) \cong F$ has the $n$-sum property with respect to $U$, by \cref{prop:infinitesiamal-quotient}, $F[G]$ has the $n$-sum property with respect to $U_{\varepsilon_{G}}.$
\end{proof}

Next, we study \cref{question:F[G]} in the other direction: namely, when $F[G]$ is semisimple or close to being semisimple. To do so, we need to introduce some terminology.  Let $H$ be a normal subgroup of $G.$ We will denote by $\varepsilon_{G,H}$ the augmentation map $F[G] \to F[G/H].$ The kernel of $\varepsilon_{G,H}$ will be denoted by $\Delta(G,H).$ 

\begin{prop} \label{prop:quotient-F[G]-semisimple}
    Suppose that $H \vartriangleleft G$ and $|H|$ is invertible in $F.$ If $F[G/H]$ has the $n$-sum property with respect to $U_{\varepsilon_{G/H}}$, then $F[G]$ has the $n$-sum property with respect to $U_{\varepsilon_{G}}.$ 
\end{prop}
\begin{proof}
    Let $e_H = \frac{1}{|H|} \sum_{h \in H}h.$ As explained  in \cite[Proposition 3.6.7]{GR}, $e_H$ is a central idempotent of $F[G]$ and we have the following isomorphism of rings 
    \[ F[G] \cong F[G]e_{H} \oplus F[G] (1-e_H).\]
    Furthermore, $F[G]e_H \cong F[G/H]$ and $F[G](1-e_H) \cong \Delta(G,H).$ Under this isomorphism, the restriction of the augmentation map $\varepsilon_{G}$ on $F[G/H]$ is nothing but $\varepsilon_{G/H}.$ On the other hand, the restriction of $\varepsilon_G$ on $\Delta(G,H)$ is $0.$ We then see that under this decomposition 
    \[ U_{\varepsilon_G} = U_{\varepsilon_{G/H}} \oplus \Delta(G,H)^{\times}.\]
    By \cref{cor:F[G]}, $F[G]$ has the $2$-sum property with respect to $F[G]^{\times}$. The projection \(F[G]\to \Delta(G,H)\) is a unital ring map onto the direct factor with identity \(1-e_H\). Hence, by \cref{cor:2-property-quotient}, $\Delta(G,H)$ also has the $2$-sum property with respect to $\Delta(G,H)^{\times}$. Since $n \geq 2$, $\Delta(G,H)$ also has the $n$-sum property with respect to $\Delta(G,H)^{\times}$ as well. By our assumption, $F[G/H]$ has the $n$-sum property with respect to $U_{\varepsilon_{G/H}}$. We conclude that $F[G]$ has the $n$-sum property with respect to $U_{\varepsilon_{G}}.$
\end{proof}
We discuss some corollaries of \cref{prop:quotient-F[G]-semisimple}. 

\begin{cor} \label{cor:semisimple}
    Suppose that $|G|$ is invertible in $F$. Then $F[G]$ has the $n$-sum property with respect to $U_{\varepsilon_G}.$
\end{cor}
\begin{proof}
    In this case, we can take $G=H$ as in \cref{prop:quotient-F[G]-semisimple}. Then, $F[G/H]=F$, which has the $n$-sum property with respect to $U$ by our assumption. 
\end{proof}

By \cref{prop:p-group} and \cref{cor:semisimple}, we have the following proposition. 

\begin{prop} \label{prop:p-group-all}
    If $G$ is a $p$-group then $F[G]$ has the $n$-sum property with respect to $U_{\epsilon_G}.$
\end{prop}

\begin{proof}
    The case $p=\text{char}(F)$ is proved in \cref{prop:p-group} and the case $p \neq \text{char}(F)$ is proved in \cref{cor:semisimple}. 
\end{proof}

We discuss another corollary where the answer to \cref{question:F[G]} is affirmative. 
\begin{cor}
 Suppose that $H \vartriangleleft G$ and $|H|$ is invertible in $F.$ Assume further that $G/H$ is a $p$-group. Then $F[G]$ has the $n$-sum property with respect to $U_{\varepsilon_{G}}$. 
\end{cor}
\begin{proof}
    This statement follows from \cref{prop:p-group-all} and \cref{prop:quotient-F[G]-semisimple}.
\end{proof}

We discuss another orthogonal variant of \cref{prop:quotient-F[G]-semisimple}.

\begin{prop} \label{prop:F[G]-p-adic-normal}
Suppose that  $H$ is a normal $p$-subgroup of $G$ where $p=\text{char}(F).$ If $F[G/H]$ has the $n$-sum property with respect to $U_{\varepsilon_{G/H}}$, then $F[G]$ has the $n$-sum property with respect to $U_{\epsilon_G}.$ In particular, if $H$ is a normal Sylow $p$-subgroup of $G$ and $F[G/H]$ has the $n$-sum property with respect to $U_{\epsilon_{G/H}}$, then $F[G]$ has the $n$-sum property with respect to $U_{\epsilon_{G}}.$
\end{prop}

\begin{proof}
    By \cite[Theorem 16.6]{passman1971infinite} and the fact that $H$ is a $p$-group, we know that $\Delta(G, H) \subset \Rad(F[G]).$ This first statement then follows from \cref{prop:infinitesiamal-quotient}.

    If \(H\) is a normal \(p\)-Sylow subgroup, then \(|G/H|\) is invertible in \(F\).
Thus Corollary 3.9 implies that \(F[G/H]\) has the \(n\)-sum property with respect
to \(U_{\varepsilon_{G/H}}\). The result follows from the first part.
\end{proof}
Other than the above scenarios, it is unclear to us whether \cref{question:F[G]} has an affirmative answer for all finite groups $G.$ We provide below a class of groups for which we know the answer to it. 

\begin{prop}
Suppose that $G$ is either a finite abelian group or a dihedral group. Then $F[G]$ has the $n$-sum property with respect to $U_{\varepsilon_{G}}$
\end{prop}

\begin{proof}
We will give a proof for the case $G=D_{2n}$ where $n \geq 1.$ The case where $G$ is abelian is similar and easier. In fact, for finite abelian \(G\), one argues by induction on \(|G|\), choosing a nontrivial normal Sylow $p$-subgroup and applying \cref{prop:quotient-F[G]-semisimple} or \cref{prop:F[G]-p-adic-normal} according as its order is invertible in \(F\) or  \(p= \operatorname{char}(F)\).

For $n=1$ or $n=2$, $D_{2n}$ is a $2$-group so the statement holds by \cref{prop:p-group-all}. Suppose that it has been proved for all $m <n$. We will show that it also holds for $D_{2n}.$ If $n$ is a power of $2$ then the statement holds by \cref{prop:p-group-all}. Otherwise, suppose that there exists $p \mid n$ such that $p \neq 2.$ Let $\alpha=v_p(n).$ Then, we can find an exact sequence of the form 
\[ 1 \to \Z/p^{\alpha} \to D_{2n} \to D_{2(n/p^{\alpha})} \to 1.\]

If $p=\text{char}(F)$, then our statement follows from \cref{prop:F[G]-p-adic-normal} and our induction hypothesis applied to $D_{2(n/p^{\alpha})}$. On the other hand, if $p \neq \text{char}(F)$, then it follows from \cref{prop:quotient-F[G]-semisimple} and the induction hypothesis applied to $D_{2(n/p^{\alpha})}.$
\end{proof}

\subsection{Sums of normalized units in a finite extension of fields}
In this section, we discuss another situation where normalized units appear quite naturally. More precisely, we are interested in the following question.

\begin{question} \label{question:normalized-units}
    Let $L/F$ be a non-trivial finite extension of finite fields. Let $N: L \to F$ be the norm map, $U_{L/F}= \{u \in L^{\times} \mid N(u)=1 \}$ and $U_{L/F}^{\pm} = \{ \pm{u} \mid u \in U_{L/F} \}.$ Does there exist a positive integer $n$ such that $L$ has the $n$-sum property with respect to $U_{L/F}^{\pm}$?
\end{question}
We remark that we need to consider $U_{L/F}^{\pm}$ to ensure that the graph $\Gamma(L, U_{L/F}^{\pm})$ is undirected. Additionally, if $[L:F]$ is even, then $N(-1)=1$ and hence $U_{L/F}= U_{L/F}^{\pm}.$


We are able to provide a complete answer to Question \ref{question:normalized-units}. In fact we are able to find the minimal $n$ for every non-trivial field extension $L/F$ of finite fields. Thus we determine all $n$ for all such field extensions $L/F$.

\begin{definition}
For a given $L/F$, a finite non-trivial extension of finite fields, we define
\[
m(L/F)=\min\{n\geq 1 : L \text{ has the } n\text{-sum property with respect to }
U_{L/F}^{\pm}\}.
\]
\end{definition}

We are able to completely determine invariants $m(L/F)$ in all cases.

\begin{thm} \label{thm:normalized-units-fields}
 Let $F=\F_q$ and $L/F$ be a non-trivial finite field extension. 
 Then $L$ is additively generated by $U_{L/F}^{\pm}$, i.e., the Cayley graph $\Gamma(L, U_{L/F}^{\pm})$ is always connected. 
 Furthermore, $m(L/F)$ has values $2,3,4$, and they occur as follows:
 \begin{enumerate}
\item $m(L/F)=2$ if $[L:F]\geq 3$, or $[L:F]=2$ and $q\in \{2,3\}$.
\item $m(L/F)=3$ if $[L:F]=2$, $q\not\in\{ 2,3\}$, and either \( q\equiv 0 \pmod 3\), or \(q\equiv 5 \pmod 6\),
or \(q=2^m\) with $m$  odd.
 \item $m(L/F)=4$ if $[L:F]=2$ and either \( q\equiv 1 \pmod 6\) or  \(q=2^m\)  with  $m$  even.
\end{enumerate}
\end{thm}
\begin{rem} \label{rem:n-sum-upper}
We observe that for a finite non-trivial extension $L/F$ and a positive integer $n$,  $L/F$ has the
$n$-sum property with respect to $U_{L/F}^{\pm}$ if and only if \( n\geq m(L/F).\) 
Indeed, it suffices to show that if $L/F$ has the $m$-sum property with
respect to $U_{L/F}^{\pm}$, then it has the $(m+1)$-sum property with
respect to $U_{L/F}^{\pm}$. Assume that $L/F$ has the $m$-sum property and
let $b\in L$. Then
\[
b-1=s_1+s_2+\cdots+s_m
\]
for some $s_i\in U_{L/F}^{\pm}$. Hence
\[
b=s_1+s_2+\cdots+s_m+1,
\]
and thus $L/F$ also has the $(m+1)$-sum property with respect to
$U_{L/F}^{\pm}$.
\end{rem}
The proof of Theorem~\ref{thm:normalized-units-fields} will follow from our further results in this
section as follows. The first statement that the Cayley graph $\Gamma(L, U_{L/F}^{\pm})$ is always connected follows  from Proposition~\ref{prop:norm-1-elements} and Proposition~\ref{prop:u-connected}

For the second statement on the possible values of $m=m(L/F)$ we first observe that because $0\notin U_{L/F}^{\pm}$, we see that $m=1$ is impossible.

If $[L:F]\geq 4$, then from Lemma~\ref{lem:[L:F]=4} it follows that $L/F$ has the $2$-sum property with respect to $U_{L/F}^{\pm}$ and therefore \( m(L/F)=2.\) 

From Theorem~\ref{thm:cubic 2-sum} it follows that if $[L:F]=3$, then $L/F$ has again the $2$-sum property and hence $m(L/F)=2$.

So the remaining case is $[L:F]=2$. Then $N_{L/F}(-1)=1$ and we see that \( U_{L/F}^{\pm}=U_{L/F}. \)
From Proposition~\ref{prop:quadratic 2-sum} we see that a quadratic extension has the $2$-sum property if and only if $q=2$ or $q=3$.

Assume now that $[L:F]=2$ and $q\not\in\{ 2,3\}$. Then Proposition~\ref{prop:quadratic 2-sum} implies
that \(m(L/F)\geq 3. \)
Proposition~\ref{prop:degree-2} implies that $L$ always has the $4$-sum property  with respect to
$U_{L/F}^{\pm}$.

Now Proposition~\ref{prop:quadratic 3-sum} says that $L/F$ has the $3$-sum property with respect to
$U_{L/F}^{\pm}$ if and only if
\[
q\equiv 0 \pmod 3,\quad q\equiv 5 \pmod 6,
\quad \text{or}\quad q=2^m \text{ with } m \text{ odd}.
\]
Because $q$ is a power of a prime $p$, the remaining cases for $q$ are
\[
q\equiv 1 \pmod 6
\quad \text{or}\quad
q=2^m \text{ with } m \text{ even}.
\]
This completes the proof of Theorem~\ref{thm:normalized-units-fields}.


Let us discuss the first statement of \cref{thm:normalized-units-fields} with the following proposition. 

\begin{prop} \label{prop:norm-1-elements}
 Let $L/F$ be a finite Galois extension.  Let $K$ be a proper subfield of $L$ and $\sigma \in \Gal(L/F)$ such that $\sigma \neq 1.$ Then, there exists $x \in L^{\times}$ such that $\dfrac{\sigma(x)}{x} \not \in K.$ In particular, since $\dfrac{\sigma(x)}{x} \in U_{L/F}$, we have  $U_{L/F} \not \subset K.$
\end{prop}

\begin{proof}

Suppose to the contrary that $\sigma(x)/x \in K$ for all $x \in L^{\times}.$ We claim that this would lead to the equality that $L=L^{\sigma} \cup K $, which is a contradiction.

For each $x$, there exists $f(x) \in K $ such that $\sigma(x)=f(x)x.$
   Similarly, we have $\sigma(x+1)=f(x+1) (x+1).$ This implies that 
   \[ (x+1)f(x+1)-xf(x)=\sigma(x+1)-\sigma(x)=1.\]
   We can rewrite this as 
   \[ x[f(x+1)-f(x)]=1-f(x+1).\]

   If $f(x+1)-f(x) \neq 0$ then $x \in K.$ Otherwise, if $f(x+1)=f(x)$, then $1-f(x+1)=0$ and hence $f(x)=1.$ This would imply that $\sigma(x)=x$; or equivalently $x \in L^{\sigma}.$ In all cases, $x \in L^{\sigma} \cup K$.       
\end{proof}

\begin{prop} \label{prop:u-connected}
   Let $L$ be a finite field and  $U$ be a subgroup of $L^{\times}$. The following conditions are equivalent 
    \begin{enumerate}
        \item $\Gamma(L, U)$ is connected. 
        \item $U \not \subset K$ for all proper subfield  $K$ of $L.$ 
    \end{enumerate}
\end{prop}

\begin{proof}
    Let $S$ be the abelian group generated by $U.$ We know that $U \subset S$. Furthermore, we can check that $S$ is a subring of $L.$ Since $L$ is a field, $S$ is even a subfield of $L.$ By our assumption, for each proper subfield $K$ of $L$, $U \not \subset K$, and hence  $S \not \subset K$.  This must imply that $S=L$. Therefore, $\Gamma(L, U)$ is connected. 

    Conversely, suppose that $\Gamma(L,U)$ is connected. Then $S=L$. Let $K$ be a subfield of $L$ such that $U \subset K.$ By the definition of $S$, $S  \subset K.$ We conclude that $K=L.$
\end{proof}

We then see that the first part of \cref{thm:normalized-units-fields} follows from \cref{prop:norm-1-elements} and \cref{prop:u-connected}.

Let us now prove the second part of \cref{thm:normalized-units-fields}. By Hilbert 90, every normalized unit in $L/F$ is of the form $\dfrac{\sigma_q(x)}{x}=x^{q-1}$ where $q=|F|$ and $\sigma_q$ is the Frobenius map $\sigma_q(x)=x^{q}.$ 
The following result says that if $[L:F]\geq 4$ then $L/F$ always has the $2$-sum property  with respect to
$U_{L/F}^{\pm}$. The proof uses some estimates for certain character sums.
\begin{lem} \label{lem:[L:F]=4}
Let $F=\F_q$ and  $[L:F] \geq 4$. Then for each $b \in L$, there exists $x,y \in L^{\times}$ such that  either $x^{q-1}-y^{q-1}=b$ or 
$x^{q-1}+y^{q-1}=b .$
\end{lem}
\begin{proof}
If $b=0$ then we can take $x=y=1$, and we have  $1^{q-1}-1^{q-1}=0.$ Let us suppose that $b \neq 0.$ We will show that we can find $x,y \in L^{\times}$ such that $x^{q-1}+y^{q-1}=b.$

Let $k=[L:F]$ and hence $L=\F_{q^k}.$
We modify the proof given in \cite[Theorem 3]{bergelson2021sums}. Let $N_2(b)$ be the number of solutions of the equation $x^{q-1}+y^{q-1}=b.$ Let $M_2(b)$ be the number of solutions such that $x, y \in L^{\times}.$ We need to show that $M_2(b) > 0$ for each $b \in L.$ For each $b$, the equation $x^{q-1}=b$ has at most $q-1$ solutions. Therefore $M_2(b) \geq N_2(b)-2(q-1).$ On the other hand, by \cite[Corollary 1, p. 57]{joly1973equations}, we have 
\[ |N_2(b)-q^k| \leq (q-2)^2 q^{k/2}< (q-1)^2q^{k/2}. \]
This shows that $N_2(b) \geq q^k - (q-1)^2 q^{k/2}$ and hence 
\[ M_2(b) \geq q^{k}-(q-1)^2 q^{k/2}-2(q-1).\]
If $k \geq 4$ then 
\[ M_2(b) \geq  [q^2-(q-1)^2] q^{k/2} - 2(q-1) =(2q-1)q^{k/2} - 2(q-1)>(2q-1)-2(q-1)=1. \]
\end{proof}
\begin{rem} \label{rem:0-sum-of-q-1}
    We can show that if $k,q$ are both odd then we cannot find $x,y \in L^{\times}$ such that $0=x^{q-1}+y^{q-1}.$ In fact, if we let $z=x/y$ then $z^{q-1}=-1.$ 
    We then have 
    \[ 1 = z^{q^k-1}=(z^{q-1})^{\frac{q^k-1}{q-1}}=-1.\]
    This would contradict the fact that $q$ is odd. Therefore, it is important to use $U_{L/F}^{\pm}$ instead of $U_{L/F}.$
\end{rem}

For the cubic case, we obtain the following result, Theorem~\ref{thm:cubic 2-sum}, which says that if $[L:F]=3$ then $L/F$ always has the $2$-sum property  with respect to $U_{L/F}^{\pm}$. 

\begin{lem}
Let $F=\mathbb{F}_q$, and let $L/F$ be a cubic extension.
Let $A,B\in F^\times$. Then there exists $w\in L$ such that
\[ N_{L/F}(w)=A,\qquad N_{L/F}(1-w)=B.
\]
\end{lem}

\begin{proof}
For each $t\in F$, consider the polynomial
\[
f_t(x) =
x^3-tx^2+(t+A+B-1)x-A \in F[x].
\]
We shall show that there exists $t\in F$ such that $f_t(x)$ is irreducible.

Since \( f_t(0)=-A\neq0,\) and 
\[ f_t(1) = 1-t+(t+A+B-1)-A = B\neq0, \]
the polynomial has no roots at $0$ or $1$.

Suppose $a\in F\setminus\{0,1\}$. Then \( a-a^2=a(1-a)\neq0.\)
If $a$ is a root of $f_t(x)$, then
\[ a^3-ta^2+(t+A+B-1)a-A=0,\]
which is equivalent to
\[ t = \frac{A-a^3-(A+B-1)a}{a-a^2}. \]
Therefore there are at most $q-2$ values of $t$ for which $f_t(x)$ has a root in $F$.
Since there are $q$ possible values of $t$, there exists at least one value of $t$ such that $f_t(x)$ has no roots in $F$.
As $f_t(x)$ is cubic, it is therefore irreducible.

Let $w$ be a root of such a polynomial $f_t(x)$. 
Since every irreducible cubic over $F$ splits over the unique degree-$3$ extension $L$, we may regard $w$ as an element of $L$.
Writing
\[ f_t(x) =  (x-w)(x-w^q)(x-w^{q^2}), \]
and comparing coefficients, we obtain
\[
\begin{aligned}
t &=w+w^q+w^{q^2},\\
t+A+B-1 &=ww^q+ww^{q^2}+w^qw^{q^2},\\
A&=ww^qw^{q^2} = N_{L/F}(w).
\end{aligned}
\]

On the other hand, we have
\begin{align*}
N_{L/F}(1-w) &= (1-w)(1-w^q)(1-w^{q^2}) \\
&= 1-(w+w^q+w^{q^2}) +(ww^q+ww^{q^2}+w^qw^{q^2}) -ww^qw^{q^2} \\
&= 1-t+(t+A+B-1)-A \\
&= B.
\end{align*}
This completes the proof.
\end{proof}

\begin{thm}
\label{thm:cubic 2-sum}
Let $F=\mathbb{F}_q$, and let $L/F$ be a cubic extension.
Then $L$ has the $2$-sum property with respect to $U_{L/F}^{\pm}$.
More precisely,

\begin{enumerate}
\item every nonzero element of $L$ is the sum of two elements of $U_{L/F}$;
\item the element $0$ is the sum of two elements of $U_{L/F}^{\pm}$.
\end{enumerate}
\end{thm}

\begin{proof}
Let $b\in L^\times$.
We want to find $u,v\in U_{L/F}$ such that
\( b=u+v. \)
Write
\[ u=bw,\qquad v=b(1-w). \]
Then
\[ N_{L/F}(u) = N_{L/F}(b)N_{L/F}(w), \]
and
\[ N_{L/F}(v) = N_{L/F}(b)N_{L/F}(1-w). \]
Let
\( c=N_{L/F}(b)\in F^\times. \)
Applying the previous lemma with
\( A=B=c^{-1}, \)
we obtain $w\in L$ satisfying
\[ N_{L/F}(w)=c^{-1}, \qquad N_{L/F}(1-w)=c^{-1}. \]
Hence \( N_{L/F}(u)=N_{L/F}(v)=1, \) so \( u,v\in U_{L/F}, \) and \( b=u+v.\)
This proves the first assertion.

The second assertion follows immediately from \( 0=1+(-1).\)
\end{proof}

Observe that the second statement cannot, in general, be strengthened by
replacing $U_{L/F}^{\pm}$ with $U_{L/F}$ when $q$ is odd.
Indeed, suppose that
\[
0=u+v,
\qquad
u,v\in U_{L/F}.
\]
Then
\( v=-u,\)
and therefore
\[
1
=
N_{L/F}(v)
=
N_{L/F}(-u)
=
(-1)^3N_{L/F}(u)
=
-1,
\]
a contradiction.

One may ask whether similar statements in \cref{thm:cubic 2-sum} hold when $[L:F]=2$. Unfortunately, they do not, but there are some patterns. For example, let $F=\mathbb{F}_4$ and $L=\mathbb{F}_{16}$. Here $[L:F]=2$ and the graph $\Gamma(L,U_{L/F}^{\pm})$ is shown in \cref{fig:GF4}. Brute-force computations in SageMath show that $L$ does not have the $3$-sum property with respect to $U_{L/F}^{\pm}$. This is also visible from \cref{fig:GF4}, where the graph $\Gamma(L,U_{L/F}^{\pm})$   has no $3$-cycles.
\begin{figure}
  \centering
  \includegraphics[width=0.4\textwidth]{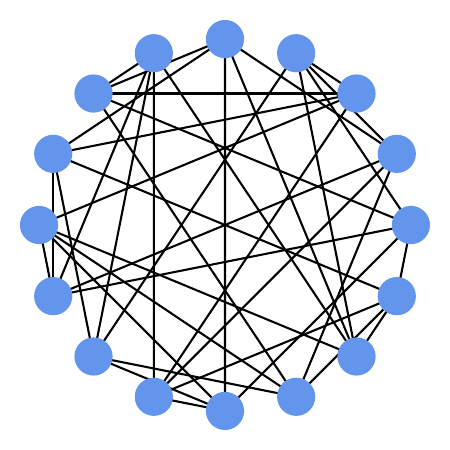}
  \caption{The Cayley graph $\Gamma(\F_{16}, U_{\F_{16}/F_4}^{\pm})$}
  \label{fig:GF4}
\end{figure}

The $n$-sum properties for $n \in \{2, 3\}$ for other values of $q$ are summarized in \cref{tab:sum-properties-transposed}.
The data suggest that when $q$ is odd, $L$ has the $3$-sum property with respect to $U_{L/F_q}^{\pm}$  only if $q \not \equiv 1 \pmod{6}.$ On the other hand, the data also suggest that the $2$-sum property with respect to $U_{L/F}^{\pm}$ fails for $q \geq 4$ (note that in this case $U_{L/F}=U_{L/F}^{\pm}$ since $N(-1)=(-1)^2=1).$ We will provide below a complete answer for these observations.
\begin{table}[htbp]
\centering
\begin{tabular}{|l|*{18}{c|}}
\hline
$q$ & 2 & 3 & 4 & 5 & 7 &  8 & 9 & 11 & 13 & 16 & 17 & 19 & 23 & 25 & 27 & 29 & 31 & 32 \\
\hline
2-sum property & T & T & F & F & F & F & F & F & F & F & F & F & F & F & F & F & F & F \\
\hline
3-sum property & T & T & F & T & F &T & T & T & F & F & T & F & T & F & T & T & F & T \\
\hline
\end{tabular}
\caption{$n$-sum property when $[L:F]=2$.}
\label{tab:sum-properties-transposed}
\end{table}




\begin{lem}
\label{lem:AS reducible}
    Let $F$ be a finite field of characteristic $p$. Let $a\in F$. The polynomial $x^p-x-a$ has a zero in $F$ if and only if ${\rm Tr}_{F/\F_p}(a)=0$.
\end{lem}

\begin{proof}
    Let $\sigma\colon F\to F$, $\alpha\mapsto \alpha^p$ be the Frobenius automorphism. Then the extension $F/\F_p$ is Galois whose Galois group is a cyclic group generated by $\sigma$. By Hilbert's Theorem 90 (the additive form), \cite[Chapter VI, Theorem 6.3]{Lang2002}, we know that ${\rm Tr}_{F/\F_p}(a)=0$ if and only if there exists $\alpha\in F$ such that
    \[
    a=\sigma \alpha-\alpha =\alpha^p-\alpha,
    \]
   i.e., if and only if the polynomial $x^p-x-a$ has a zero in $F$.
\end{proof}

\begin{lem}
\label{lem:U+U}
    Let $F=\F_q$ and  $L=\F_{q^2}$ be a quadratic extension of $F$. Let $b\in L^\times$.
    \begin{itemize}
     \item[(a)] If  ${\rm char} F\not=2$ then $b\in U_{L/F}+U_{L/F}$ if and only if  $N(b)(N(b)-4)=0$ or a nonsquare.
        \item[(b)] If  ${\rm char} F=2$ then $b\in U_{L/F}+U_{L/F}$ if and only if  ${\rm Tr}_{F/\F_2}\left(\dfrac{1}{N(b)}\right)=1$.
           \end{itemize}
\end{lem}

\begin{proof}
    Suppose $b\in U_{L/F}+U_{L/F}$. We can write $b=u+v$, where $u\in U_{L/F}, v\in U_{L/F}$. Write
    \[
    u=bw,\quad v=b(1-w), \quad w\in L.
    \]
    We have $N(w)=N(1-w)=\dfrac{1}{N(b)}$. Now
    \[
  N(1-w) = (1-w) (1-w)^q=(1-w)(1-w^q)=1-(w+w^q)+w^{q+1}=1-{\rm Tr}(w) + N(w).
    \]
    This implies  ${\rm Tr}(w)=1$.

   (a) First, we consider the case where $q$ is odd. If $w\in F$ then from ${\rm Tr}(w)=1$ we see that $w=1/2$ and hence $N(b)=1/N(w)=4$. We assume that $w\not\in F^\times$. Since $w$ is a root of equation $X^2- X +\dfrac{1}{N(b)}=0$, we conclude that $\Delta=1-\dfrac{4}{N(b)}$ is not a perfect square in $F^\times$. Hence $N(b)(N(b)-4)\not\in (F^\times)^2$.

     Conversely, if $N(b)=4$ then let $u=v=b/2$ then $N(u)=N(v)=1$ and $b=u+v$. Hence $b\in U_{L/F}+U_{L/F}$. If $N(b)(N(b)-4)\not\in (F^\times)^2$ then
      $\Delta=1-\dfrac{4}{N(b)}$ is  not a perfect square in $F^\times$. Let $w$ be a root of the irreducible polynomial $X^2- X +\dfrac{1}{N(b)}$. Hence ${\rm Tr}(w)=1$ and $N(w)=\dfrac{1}{N(b)}$. Thus
     \[
     N(1-w)= 1-{\rm Tr}(w) + N(w)=N(w)=\dfrac{1}{N(b)}.
     \]
     Let $u=bw$ and $v=b(1-w)$. Then $N(u)=N(v)=1$ and hence $u\in U_{L/F}$, $v\in U_{L/F}$. Clearly $b=u+v\in U_{L/F}+U_{L/F}$.

(b)    Now we consider the case where $q$ is even. Since ${\rm Tr}(w)=1$ and $N(w)=\dfrac{1}{N(b)}$, $w$ is not in $F$ and the minimal polynomial of $w$ over $F$ is 
    \[
    X^2+X + \dfrac{1}{N(b)}.
    \]
    This polynomial is irreducible over $F$, therefore it has no zero in $F$. By the previous lemma ${\rm Tr}_{F/\F_2}(\dfrac{1}{N(b)})\not=0$. We obtain that 
    \({\rm Tr}_{F/\F_2}\left(\dfrac{1}{N(b)}\right)=1. \)
    
        Conversely, suppose ${\rm Tr}_{F/\F_2}\left(\dfrac{1}{N(b)}\right)=1$. Then the polynomial  $X^2+X + \dfrac{1}{N(b)}$ is irreducible over $F$. We choose $w\in L$ such that 
    \[
    N(w)=\dfrac{1}{N(b)},\quad {\rm Tr}(w)=1,
    \]
    i.e., we choose $w\in L$ as one of the roots of equation $X^2-X+\dfrac{1}{N(b)}=0$. Then $N(1-w)=1-{\rm Tr}(w)+N(w)=N(w)=\dfrac{1}{N(b)}$. Let $u=bw$ and $v=b(1-w)$. Then $N(u)=N(v)=1$ and hence $u\in U_{L/F}$, $v\in U_{L/F}$. Clearly $b=u+v\in U_{L/F}+U_{L/F}$.
\end{proof}

\begin{prop}
\label{prop:quadratic 2-sum}
Let $F=\F_q$ and $L=\F_{q^2}$ be a quadratic extension of $F$. Then $L$ has the 2-sum property with respect to $U_{L/F}^{\pm}$ if and only if $q=2$ or $q=3$.
\end{prop}
\begin{proof}
If  \(F=\mathbb{F}_2\), then  \(L=\mathbb{F}_4 =\{0,1,\alpha,1+\alpha\},\alpha^2+\alpha+1=0. \)
We have \( U_{L/F}=L^\times.\)
Furthermore,
\[
0=1+1,\quad
1=\alpha+(\alpha+1),\quad
\alpha=1+(\alpha+1),\quad
\alpha+1=\alpha+1.
\]

If $F=\mathbb{F}_3$, then \(L=\mathbb{F}_9=\{a+bi:a,b\in\mathbb{F}_3\},\quad i^2=-1. \)
Then \(U_{L/F}=\{\pm 1,\pm i\}.\)
Hence
\[
0=1+(-1),\quad
1=(-1)+(-1),\quad
-1=1+1,
\]
\[
i=(-i)+(-i),\quad
-i=i+i,
\]
\[
1+i=1+i,\quad
1-i=1+(-i),\quad
-1+i=-1+i,\quad
-1-i=-1-i.
\]
Thus for $q=2$ or $q=3$, $L$ has the $2$-sum property with respect to $U_{L/F}^{\pm}=U_{L/F}$.

Assume now that $q\geq 4$.  Since $[L:F]=2$, $U_{L/F}^{\pm}=U_{L/F}=:U=\{u\in L^\times: N(u)=1\}$. 
    
  First, we consider the case where $q$ is odd. We choose any $a$, which is a nonzero square in $F$ and $a\not=1/4$ ({this is possible since $q>3$}). Since the norm map $N\colon L^\times \to F^\times$ is surjective, we can choose $b\in L^\times$ such that $N(b)=\dfrac{4}{1-4a}$. Clearly $N(b)\not=4$ and $1-\dfrac{4}{N(b)}=4a\in (F^\times)^2$. Hence by Lemma~\ref{lem:U+U}, $b$ is not in $U_{L/F}+U_{L/F}$. Hence $L$ does not have the 2-sum property with respect to $U_{L/F}^{\pm}$.

    Now we consider the case where $q$ is even. 
    Since $q\geq 4$, there is a nonzero $c\in F$ such that ${\rm Tr}_{F/\F_2}(c)=0$. (The map ${\rm Tr}_{F/\F_2}\colon F\to \F_2$ is  $\F_2$-linear.) Now we choose $b\in L^\times$ such that $N(b)=\dfrac{1}{c}$. Clearly for such an element $b$, we have $ {\rm Tr}_{F/\F_2}\left(\dfrac{1}{N(b)}\right)={\rm Tr}_{F/\F_2}(c)=0$. Hence by Lemma~\ref{lem:U+U}, $b$ is not in $U_{L/F}+U_{L/F}$. Hence $L$ does not have the 2-sum property with respect to $U_{L/F}^{\pm}$.
\end{proof}

\begin{lem} 
\label{lem:N(b-u)=d}
 Let $F=\F_q$ and  $L=\F_{q^2}$ be a quadratic extension of $F$. Let $b\in L^\times$, $c=N(b)\in F^\times$ and  $d\in F^\times$. The following conditions are equivalent.
 \begin{enumerate}
     \item There exists $u\in U_{L/F}$ such that $N(b-u)=d$.
     \item There exists $w\in L^\times$ such that ${\rm Tr}(w)=c+1-d$ and $N(w)=c$.
 \end{enumerate} If  $q$ is odd then these equivalent conditions (1) and (2) are also equivalent to 
 \begin{enumerate}
     \item[(3)] $(c+1-d)^2-4c$ is  0 or a nonsquare.
 \end{enumerate}
If  $q$ is even then these equivalent conditions (1) and (2) are also equivalent to 
 \begin{enumerate}
     \item[(4)] $(c+1-d)=0$ or ${\rm Tr}_{F/\F_2}(\dfrac{c}{(c+1-d)^2})=1$.
 \end{enumerate}
\end{lem}
\begin{proof} (1) $\Rightarrow (2)$:     Suppose that there exists $u\in U_{L/F}$ such that $N(b-u)=d$. We have
    \[d=N(b-u)=(b-u)(b^q-u^q)= b^{q+1} +u^{q+1}- ub^q-u^qb=c+1- ub^q-u^qb.\]
   Note that $u^{q+1}=1$ hence  \[c+1-d= ub^q+u^qb= (u^qb)^q+u^qb={\rm Tr}(u^qb).\]
   Set $w:=u^qb=u^{-1}b\in L^\times$. Then ${\rm Tr}(w)=c+1-d$ and  
    $N(w)=N(u^qb)=N(b)=c$. 

    (2) $\Rightarrow (1)$:  Suppose that there exists $w\in L^\times$ such that ${\rm Tr}(w)=c+1-d$ and $N(w)=c$. Set $u=b/w$. Then $N(u)=1$, i.e., $u\in U_{L/F}$, and $w=u^{-1}b=u^qb$.
   We have
   \[
   N(b-u)=(b-u)(b^q-u^q)= b^{q+1} +u^{q+1}- (ub^q+u^qb)=c+1-{\rm Tr}(w)=d.
   \]

    We assume further that $q$ is odd. 

    (2) $\Rightarrow$ (3): Suppose that there exists $w\in L^\times$ such that ${\rm Tr}(w)=c+1-d$ and $N(w)=c$.
   Then $w$ is a root of equation $X^2-(c+1-d)X+c=0$. If $w\in F^\times$ then ${\rm Tr}(w)=2w=c+1-d$ and $c=N(w)=w^2=(c+1-d)^2/4$. Hence $(c+1-d)^2-4c=0$. If $w\not\in F^\times$ then the discriminant $(c+1-d)^2-4c\not\in (F^\times)^2$.

   (3) $\Rightarrow$ (2): Suppose that $(c+1-d)^2-4c$ is  0 or a nonsquare. If $(c+1-d)^2-4c=0$, set $w=(c+1-d)/2\in F^\times$. Then ${\rm Tr}(w)=c+1-d$ and $N(w)=w^2=(c+1-d)^2/4=c$. 
   If $(c+1-d)^2-4c$ is nonsquare then $X^2-(c+1-d)X+c$ is irreducible over $F$. Let $w$ be a root of $X^2-(c+1-d)X+c=0$.  
 Then ${\rm Tr}(w)=c+1-d$ and $N(w)=c$.

 Now we assume that $q$ is even.
 
  (2) $\Rightarrow$ (4): Suppose that there exists $w\in L^\times$ such that ${\rm Tr}(w)=c+1-d$ and $N(w)=c$.
   Then $w$ is a root of equation $X^2-(c+1-d)X+c=0$. 
      If $w\in F^\times$ then $(c+1-d)={\rm Tr}(w)=2w=0$.  If $w\not\in F^\times$ then polynomial $X^2-(c+1-d)X+c$ is irreducible over $F$. In particular $c+1-d\not=0$. (If $c+1-d=0$ then since $F$ is a perfect field, $c=w^2$ for some $w\in F$ and $X^2+c=(X+w)^2$ is reducible.) Divide the polynomial $X^2-(c+1-d)X+c$ by $(c+1-d)^2$ and set $Y=X/(c+1-d)$. Then $w/(c+1-d)$ is a root of $Y^2-Y+\dfrac{c}{(c+1-d)^2}$. Hence this polynomial is irreducible. By Lemma~\ref{lem:AS reducible}, ${\rm Tr}_{F/\F_2}(\dfrac{c}{(c+1-d)^2})=1.     $

   (4) $\Rightarrow$ (2): If $c+1-d=0$  then argue as before we can write $c=w^2$ for some $w\in F$ and $X^2+c=(X+w)^2$. We have ${\rm Tr}(w)=0=c+1-d$ and $N(w)=w^2=c$. 
   
If   ${\rm Tr}_{F/\F_2}(\dfrac{c}{(c+1-d)^2})=1$ then $X^2-(c+1-d)X+c$ is irreducible over $F$. Let $w$ be a root of $X^2-(c+1-d)X+c=0$.  
 Then ${\rm Tr}(w)=c+1-d$ and $N(w)=c$. 
\end{proof}
\begin{lem} \label{lem:q=1 mod 6} Let $F=\F_q$ and  $L=\F_{q^2}$ be a quadratic extension of $F$.
     The equation 
    \[ u_1+u_2=1, \]
    has a solution such that $u_1, u_2 \in U^\pm_{L/F}$ if and only if  $q\equiv 0\pmod3$ or  $q \equiv 5 \pmod{6}$ or $q=2^m$ with $m$ is odd.  Similarly, the equation 
    \[   u_1+u_2+u_3=0, \]
has a solution such that $u_1, u_2, u_3 \in U^\pm_{L/F}$ if and only if $q\equiv 0\pmod3$ or $q \equiv 5 \pmod{6}$ or $q=2^m$ with $m$ is odd.
\end{lem}

\begin{proof}
We have that 
\[
U_{L/F}^\pm = U_{L/F}=\{x\in \F_{q^2}^\times: N(x)=1\}=\{x^{q-1}: x\in \F_{q^2}^\times\}.
\]
Since $\F_{q^2}^\times$ is a cyclic group of order $q^2-1$,  $U_{L/F}$ is the unique cyclic group of order $q+1$ of $\F_{q^2}^\times$.

  Suppose $q\equiv 0\pmod 3$. Let $u_1=u_2=-1\in U_{L/F}$. Clearly $1=u_1+u_2$.

  Suppose that $q\equiv 5\pmod 6$. Then $6\mid q+1$. Thus there exists $u_1$ in the cyclic group $U_{L/F}$ such that $u_1$ is of order 6. This implies that $u_1$ satisfies the equation $X^2-X+1=0$. Let $u_2=\dfrac{1}{u_1}$. Then $u_2\in U_{L/F}$ and $u_1+u_2=1$, as desired.

Suppose that $q=2^m$ with $m$ odd. Then $3\mid q+1$. Thus there exists $u_1$ in the cyclic group $U_{L/F}$ such that $u_1$ is of order 3. This implies that $u_1$ satisfies the equation $X^2-X+1=X^2+X+1=0$. Let $u_2=\dfrac{1}{u_1}$. Then $u_2\in U_{L/F}$ and $u_1+u_2=1$, as desired.

We show that for the two remaining cases where $q\equiv 1\pmod 6$ or $q=2^m$  with $m$ even, the equation $1=u_1+u_2$ has no solution $u_1,u_2\in U_{L/F}$. Assume for a contradiction that $1=u_1+u_2$, where $u_1,u_2\in U_{L/F}$.
   Then, applying the Frobenius automorphism $\sigma_q$ to both sides, we get 
    \[1=u_1^{q}+u_2^{q}=u_1^{-1}+u_2^{-1}.\]
    Thus $u_1u_2=u_1+u_2=1$. 
    Therefore, $u_1, u_2$ are roots of the equation $T^2-T+1=0.$

   If $q\equiv 1\pmod 6$  then $u_1$ is an element of order 6 in the cyclic subgroup $U_{L/F}$.
    Thus $6\mid q+1$, which contradicts the condition $q\equiv 1\pmod 6$.

   If $q=2^m$  with $m$ even then $q\equiv 1\pmod 3$  and $u_1$ is an element of order 3  in the cyclic subgroup $U_{L/F}$.
    Thus $3\mid q+1$, which contradicts the condition $q\equiv 1\pmod 3$.

    For the second equation, by dividing both sides by $-u_3$ and noting that $-1\in U_{L/F}$ we get the equation $u'_1+u'_2=1$, where $u'_1=\dfrac{u_1}{-u_3}\in U_{L/F}$ and $u'_2=\dfrac{u_2}{-u_3}\in U_{L/F}$. The result then follows from the first part. 
\end{proof}

\begin{lem}     Let $F=\F_q$, where $q$ is odd, and let $c\in F^\times$. 
Set
\[
A=\{d\in F: (c+1-d)^2-4c=0 \text{ or nonsquare}\},
\]
and 
\[
B=\{d\in F: d(d-4)=0 \text{ or nonsquare}\}
\]
Then \[
|A|=\begin{cases} \dfrac{q+1}{2} & \text{ if $c$ is a nonsquare} \\
\dfrac{q+3}{2} & \text{ if $c$  is a square,}
\end{cases}
\]
and $|B|=\dfrac{q+3}{2}$.    
\end{lem}
\begin{proof}
    Let $\chi$ be the quadratic character of $\F_q$, with $\chi(0)=0$. 
    We first count the number of elements in $B$. By \cite[Theorem 5.48]{LN1997}, we have
    \[
    \sum_{d\in F} \chi( d(d-4))=-1.
    \]
    Let $m$ and $n$  be the number of $d$ such that $d(d-4)$ is, respectively, nonzero square and nonsquare.  We have
    \[ m+n=q-2,\qquad m-n=-1.\]
    Hence $n=\dfrac{q-1}{2}$. Adding two values $d=0,4$ for which $d(d-4)=0$ we get $|B|=\dfrac{q+3}{2}$.

    Now we count the number of elements in $A$.  Let $a=c+1-d$. We see that 
    \[
    |A| =\# \{a\in F: a^2-4c=0 \text{ or nonsquare}\}.
    \]
    Again by \cite[Theorem 5.48]{LN1997}, we have
    \[
       \sum_{a\in F} \chi( a^2-4c)=-1.
    \]
    Let $s$ be the number of roots in $F$ of $a^2-4c=0$. Clearly $s=0$ if $c$ is nonsquare, and $s=2$ if $c$ is a nonzero square. 
     Let $m$ and $n$  be the number of $a$ such that $a^2-4c$ is, respectively, nonzero square and nonsquare.  We have
    \[ m+n=q-s,\qquad m-n=-1.\]
    Hence $n=\dfrac{q-s+1}{2}$. Adding $s$ values of $a$ for which $a^2-4c=0$ we get $|A|=\dfrac{q+s+1}{2}$.
\end{proof}
\begin{lem}
\label{lem:3-sum}
    Let $F=\F_q$ and  $L=\F_{q^2}$ be a quadratic extension of $F$. If $b\in L^\times$ then $b\in U_{L/F}+U_{L/F}+U_{L/F}$.
\end{lem}
\begin{proof}
We first suppose that $q=2$. Then $L=\F_4=\{0,1,\alpha,1+\alpha\}$, where $\alpha^2+\alpha+1=0$. Clearly $N(1)=N(\alpha)=N(1+\alpha)=1$, i.e, every nonzero element in $L$ is $U_{L/F}$. Hence for each $b\in L^\times$, $b=b+b+b\in  U_{L/F}+U_{L/F}+U_{L/F}$.

Now we suppose that $q\geq 3$. Let $b\in L^\times$ and let $c=N(b)$. The idea is that we can choose $u\in U_{L/F}$ such that $b-u\in U_{L/F}+U_{L/F}.$

First we consider the case $q$ is odd. Set
\[
A=\{d\in F: (c+1-d)^2-4c=0 \text{ or nonsquare}\},
\]
and 
\[
B=\{d\in F: d(d-4)=0 \text{ or nonsquare}\}
\]
By the previous lemma, we see that 
 \[|A|+|B|=\dfrac{q+3}{2} + \frac{q+s+1}{2}=q+2 +\dfrac{s}{2}\geq q+2,\]
 where $s=0$ if $c$ is a nonsquare and $s=2$ if $c$ is square. 
Hence $|A\cap B|=|A|+|B| -|A\cup B|\geq 2$. Thus we can  choose a nonzero $d\in A\cap B$. Since $d\in A$ and $d\not=0$, by Lemma~\ref{lem:N(b-u)=d}, there exists $u\in U_{L/F}$ such that $N(b-u)=d$. By Lemma~\ref{lem:U+U}, $b-u\in U_{L/F}+U_{L/F}$ and we are done.

Now we consider the case where $q=2^m$ is even, $m\geq 2$. Since $F$ is perfect, we can write $c=e^2$, $e\in F^\times$. 
Set
\[A=\{d\in F: d\not=0, {\rm Tr}_{F/\F_2}(1/d)=1\},
\]
and
\[
B=\{ d\in F: d\not=c+1, {\rm Tr}_{F/\F_2}(\dfrac{e}{c+1-d})=1\}.
\]
Then $|A|=|B|=q/2$. In fact, since ${\rm Tr}_{F/\F_2}\colon F\to \F_2$ is a surjective group homomorphism, the number of $d'$ such that ${\rm Tr}_{F/\F_2}(d')=1$ is $q/2$. For any such $d'$, $d'\not=0$ and by letting $d=1/d'$, we see that $|A|=q/2$. Clearly, $d\in B$ if and only if $(c+1-d)/e\in A$ if and only if 
\[d\in (c+1) -eA:= \{c+1 -ea: a\in A\}.\]
Hence $B= (c+1)-eA$ and $|B|=q/2$. 

We claim that $\sum\limits_{d\in A} d=1$. In fact, consider the polynomial
\[
P(X) :={\rm Tr}(X)+1 := 1+ X+X^2+X^4+\cdots+ X^{2^{m-1}}.
\]
The set of roots of $P(X)$ is exactly the set 
\[A':=\{d'\in F: {\rm Tr}_{F/\F_2}(d')=1\}=\{1/d: d\in A\}.\]
We have $P(X)=\prod\limits_{d'\in A'}(X-d')$. From Vieta, we see that $\prod_{d'\in A'} d'=1$ and 
\[
\sum_{d'\in A'} \dfrac{1}{d'}\prod_{f\in A'}f=1. 
\]
Hence 
\[
1=\sum_{d'\in A'} \dfrac{1}{d'}=\sum_{d\in A} d. 
\]
Thus
\[
\sum_{d\in B} d=\sum_{d\in A} (c+1)-ed =(c+1)|A|- e\sum_{d\in A} d=-e.
\]
Here we note that $|A|=2^{m-1}$ is even and hence $(c+1)|A|=0$. 

Now we suppose that $A\cap B=\emptyset$. Since $|A|=|B|=q/2$, we see that $F$ is the disjoint union of $A$ and $B$. Hence
\[
0=\sum_{d\in F}d= \sum_{d\in A}d +\sum_{d\in B}d= 1-e.
\]
This forces $e=1$ and $c=e^2=1$. Thus $B=(c+1)-eA=-A=A$. This contradicts $A\cap B=\emptyset$.

Therefore, there exists $d\in A\cap B$. Note that for any $x\in F$, we have ${\rm Tr}_{F/\F_2}(x)={\rm Tr}_{F/\F_2}(x^2)$. Since $d\in B$, $1={\rm Tr}_{F/\F_2}(\dfrac{e^2}{(c+1-d)^2})={\rm Tr}_{F/\F_2}(\dfrac{c}{(c+1-d)^2})$. By Lemma~\ref{lem:N(b-u)=d}, there exists $u\in U_{L/F}$ such that $N(b-u)=d$. By Lemma~\ref{lem:U+U}, $b-u\in U_{L/F}+U_{L/F}$ and we are done.
\end{proof} 

The following two propositions completely answer the $n$-sum property of $L$ with respect to $U^\pm_{L/F}$ when $[L:F]=2.$ First, we deal with the $3$-sum property.

\begin{prop}
\label{prop:quadratic 3-sum}
    Let $F=\F_q$, and  $L=\F_{q^2}$ be a quadratic extension of $F$. Then $L$ has the 3-sum property with respect to $U^\pm_{L/F}$ if and only if $q\equiv 0\pmod 3$ or $q \equiv 5 \pmod{6}$ or $q=2^m$ with $m$ is odd.
\end{prop}
\begin{proof} This follows from Lemma~\ref{lem:3-sum} and Lemma~\ref{lem:q=1 mod 6}.
\end{proof}
Next, we show that, in all cases the $4$-sum property holds. 
\begin{prop} \label{prop:degree-2}
    Let $F=\F_q$, and  $L=\F_{q^2}$ be a quadratic extension of $F$. Then $L$ always has the 4-sum property with respect to $U^\pm_{L/F}$.
\end{prop}
\begin{proof} We first suppose that $q=2$.  Then  $L=\F_4=\{0,1,\alpha,1+\alpha\}$, where $\alpha^2+\alpha+1=0$. We have $U_{L/F}=L^\times$ and 
\[
0=1+1+1+1, \;1= \alpha^2+\alpha+1 +1,\; \alpha= \alpha^2+\alpha+1 +\alpha,\; \alpha^2= \alpha^2+\alpha+1 +\alpha^2.
\]
Hence $L$ has the 4-sum property with respect to $U^\pm_{L/F}$.

Now suppose that $q\geq 3$. Let $b\in L$ be an arbitrary element. Since $U^\pm_{L/F}=U_{L/F}$ is a cyclic group of order $q+1>2$, we can choose $u_1\in U_{L/F}$ such that $b-u_1\not=0$. By Lemma~\ref{lem:3-sum}, $b-u_1$ is in $U_{L/F}+U_{L/F}+U_{L/F}$, and we are done.
\end{proof}
As we explained after Theorem~\ref{thm:normalized-units-fields}, we have now completed a sharp answer to Question~\ref{question:normalized-units} providing a list of all $n$ for all finite nontrivial extensions of finite fields.

\section{Sums of units and perfect state transfer on graphs} \label{sec:pst}
Let $G$ be an undirected simple graph with adjacency matrix $A_G$. The continuous-time quantum walk on $G$ is given by $F(t) = \exp(\mathrm{i}A_G t)$. Perfect state transfer (PST) occurs in $G$ when there exist distinct vertices $a$ and $b$ and some positive real number $t$ such that $|F(t)_{ab}| = 1$. This phenomenon is first studied in \cite{christandl2004perfect} within the framework of quantum spin networks. Following this foundational work, numerous articles have investigated PST on arithmetic graphs (see \cite{bavsic2009perfect, cheung2011perfect, godsil2012state, saxena2007parameters} among others in this research direction). It is known that any regular graph exhibiting PST must be integral—that is, all its eigenvalues are integers (see \cite{godsil2012state, saxena2007parameters}). Cayley graphs are, therefore, particularly relevant because there is a rather complete classification of integral Cayley graphs (see \cite{godsil2025integral, klotz2007some, nguyen2024integral, so2006integral}). In this section, we study the relationship between the \textit{non-existence} of PST on a Cayley graph defined over a ring $R$ and the $n$-sum property of $R.$ To do spectral analysis, we will restrict ourselves to a class of rings called the $\Z/m$-symmetric Frobenius rings (see \cite{honold2001characterization, lamprecht1953allgemeine}). We remark that every finite ring is an algebra over $\Z/m$ for some $m$ (for example, $m$ could be the characteristic of $R$). This kind of ring was first studied by Lamprecht in his investigation into Gauss sums and Ramanujan sums. We recall that a $\Z/m$-algebra $R$ is called a finite symmetric Frobenius ring if it is equipped with a linear functional $\psi: R \to \Z/m$ such that 
\begin{enumerate}
    \item $\psi(ab)=\psi(ba)$ for $a,b \in R.$
    \item $\psi$ is non-degenerate, meaning that  the kernel of $\psi$ does not contain any left-ideal in $R.$
\end{enumerate}
Associated to $\psi$, we can define a complex-valued character $\chi$ of $(R,+)$ by the rule $\chi(a)= \zeta_m^{\psi(a)}$ where $\zeta_m$ is a fixed primitive $m$-th root of unity.  The non-degenerate property of $\psi$ implies that each character of $(R,+)$ is of the form $\chi_r$ where $\chi_r: R \to \C^{\times}$ such that $\chi_r(a)=\chi(ra).$
All examples that we discussed in \cref{sec:normalized_units} are symmetric Frobenius rings. Furthermore, all finite semisimple rings are symmetric Frobenius rings (see \cite{honold2001characterization}). For the rest of this section, we will assume that $R$ is a finite $\Z/m$-symmetric algebra.

For an undirected Cayley graph $\Gamma(R,S)$ over $R$, its spectrum is parametrized by $R$. More precisely, the spectrum is the multiset  $\{\lambda_r\}_{r \in R}$ 
where 
\[ \lambda_r = \sum_{s \in S} \chi_r(s) = \sum_{s \in S} \chi(rs) = \sum_{s \in S} \zeta_m^{\psi(rs)}.\]

In \cite{nguyen2026supercharacters}, following the method of \cite{bavsic2009perfect},  we prove the following 
\begin{prop} (\cite[Theorem 4.1]{nguyen2026supercharacters}) \label{prop:orthogonal}
Let $s \in R \setminus \{0\}.$ There exists perfect state transfer from $0$ to $s$ at time $t$ if and only if for all $r_1, r_2 \in R$
    \[ (\lambda_{r_1}-\lambda_{r_2}) \frac{t}{2 \pi} + \frac{\psi((r_1-r_2)s)}{m} \equiv 0 \pmod{1}.\]
    
\end{prop}

Let us now explain the connection between the existence of PST on a Cayley graph and the $n$-sum property. 
Let $\Delta$ be the abelian group generated by the differences $r_1-r_2$ such that $\lambda_{r_1}=\lambda_{r_2}.$ \cref{prop:orthogonal} implies that if there exists a PST between $0$ and $s$ at time $t$, then $\psi(ds)=0$ for each $d \in \Delta. $ In particular, if $\Delta=R$, then the non-degeneracy of $\psi$ implies that $s=0$, which is a contradiction. 

Let $U$ be a subgroup of $R^{\times}$ such that $-1 \in U.$ The Cayley graph $\Gamma(R, S)$ is called an $U$-unitary Cayley graph if $S$ is stable under the left and right actions of $U$ (when $U=R^{\times}$, we recover the definition of a gcd-graph as explained in \cref{sec:sum_units}.) In \cite[Proposition 3.6]{nguyen2026supercharacters}, using supercharacter theory, we show that if $r_1 = u_1r_2 u_2$ where $u_1, u_2 \in U$, then $\lambda_{r_1}=\lambda_{r_2}.$
Consequently, if we denote $\Delta_U$ to be the abelian group generated by $r_1-r_2$ where $r_1, r_2 \in U$, then $\Delta_U \subset \Delta.$ We then have the following theorem. 

\begin{thm} \label{thm:pst-and-n-sum}
    Let $U \subset R^{\times}$ such that $-1 \in U.$ Suppose that $R$ has the $n$-sum property with respect to $U.$ Let $\Gamma(R,S)$ be an $U$-unitary Cayley graph. Then $\Gamma(R,S)$ has no PST. 
\end{thm}

\begin{proof}
We claim that $\Delta_U=R.$ First, we observe that since \(1\in U\), the \(n\)-sum property implies the \(N\)-sum property for every
\(N\geq n\). In fact, if \(a-1\) is a sum of \(N-1\) elements of \(U\), then \(a\) is
obtained by adding \(1\). In particular, there exists a positive integer $m$ such that $R$ has the $2m$-sum property with respect to $U$ (we can take $m= \lceil  n/2 \rceil.$) We claim that $\Delta_U=R$. Let \(a\in R\). Since \(R\) has the \(2m\)-sum property with respect to \(U\), we can write
\[
a=u_1+\cdots+u_{2m}
\]
with \(u_i\in U\). Since \(-1\in U\), we have
\[
u_{2j-1}+u_{2j}=u_{2j-1}-(-u_{2j}),
\]
where both \(u_{2j-1}\) and \(-u_{2j}\) belong to \(U\). Thus each pair
\(u_{2j-1}+u_{2j}\) lies in \(\Delta_U\), and hence \(a\in\Delta_U\). Therefore
\(\Delta_U=R\).
 Since $\Delta_U \subset \Delta$, this shows that $\Delta =R $ as well. Consequently, there is no PST on $\Gamma(R,S).$
\end{proof}
We have the following immediate corollary. 
\begin{cor}
    Let $R$ be a finite ring with the $2$-sum property. Let $\Gamma(R,S)$ be a gcd-graph over $R.$ Then $\Gamma(R,S)$ has no PST. 
\end{cor}

Finally, we show that there is no PST on the graph $\Gamma(L, U_{L/F}^{\pm})$ studied in \cref{sec:normalized_units}.
\begin{prop}
    Let $L/F$ be a finite extension of finite fields such that $L \neq F$. Then, the graph $\Gamma(L, U_{L/F}^{\pm})$ has no PST. 
\end{prop}

\begin{proof}
    If $[L:F] \geq 2$, then $L$ has the $4$-sum property with respect to $U_{L/F}^{\pm}$ by \cref{thm:normalized-units-fields}, and \cref{rem:n-sum-upper}. Therefore, by \cref{thm:pst-and-n-sum}, $\Gamma(L, U_{L/F}^{\pm})$ has no PST.

    We will present below another argument which is somewhat more direct.  Let $\Delta_2$ be the abelian group generated by elements of the form $r_1-r_2$ where $r_1, r_2 \in L^{\times}$ and $N_{L/F}(r_1)=N_{L/F}(r_2).$ By definition, $r_1 = ur_2$ where $N(u)=1$ and hence, $u \in U_{L/F}^{\pm}.$ We conclude that $r_1$ and $r_2$ give rise to the same eigenvalue of $\Gamma(L, U_{L/F}^{\pm})$; namely $\lambda_{r_1}=\lambda_{r_2}.$ This shows that $\Delta_2 \subset \Delta.$ We claim that $\Delta_2$ is an ideal in $L.$ Clearly, $\Delta_2$ is closed under addition. Let us now show that $a\Delta_2 \subset \Delta_2$ for each $a \in L^{\times}$. In fact, we have $a(r_1-r_2) =(ar_1)-(ar_2).$ Since $N(ar_1)=N(ar_2)$, we conclude that $a(r_1-r_2) \in \Delta_2$ and hence $a \Delta_2 \subset \Delta_2.$ This shows that $\Delta_2$ is an ideal in $L$. However, since $L$ is a field, we must have $\Delta_2=L$ or $\Delta_2=0$. Because \(L/F\) is nontrivial,  \(U_{L/F}\) has more than one element and  therefore \(\Delta_2\neq 0\). We conclude that $\Delta=\Delta_2=L.$ Consequently, there is no PST on $\Gamma(L, U_{L/F}^{\pm}).$    
\end{proof}

\section*{Acknowledgements}
We thank Sunil Chebolu for alerting us to the work \cite{invertible_matrices}, which initiated our study in this line of research. We thank Professor David Anderson for his insightful correspondence on his work on total graphs.  We acknowledge that we consulted OpenAI for a critique of  our exposition, for checking our arguments and in a few cases for pointing out some new slick technical ideas. We take full responsibility for the content and accuracy.


\begin{thebibliography}{10}

\bibitem{unitary}
R.~Akhtar, M.~Boggess, T.~Jackson-Henderson, I.~Jim{\'e}nez, R.~Karpman, A.~Kinzel, and D.~Pritikin, \emph{On the unitary {Cayley} graph of a finite ring}, Electron. J. Combin. \textbf{16} (2009), no.~1, Research Paper 117, 13 pages.

\bibitem{Alluqmani2026}
Eman Alluqmani, \emph{Unit graphs of group rings}, AIMS Mathematics \textbf{11} (2026), no.~1, 2907--2934.

\bibitem{anderson2008total}
David~F Anderson and Ayman Badawi, \emph{The total graph of a commutative ring}, Journal of algebra \textbf{320} (2008), no.~7, 2706--2719.

\bibitem{Ashrafi2010}
N.~Ashrafi, H.~R. Maimani, M.~R. Pournaki, and S.~Yassemi, \emph{Unit graphs associated with rings}, Comm. Algebra \textbf{38} (2010), no.~8, 2851--2871. \MR{2730284}

\bibitem{bavsic2009perfect}
M.~Ba{\v{s}}i{\'c}, M.~D. Petkovi{\'c}, and D.~Stevanovi{\'c}, \emph{Perfect state transfer in integral circulant graphs}, Applied Mathematics Letters \textbf{22} (2009), no.~7, 1117--1121.

\bibitem{bergelson2021sums}
Vitaly Bergelson, Andrew Best, and Alex Iosevich, \emph{Sums of powers in large finite fields: a mix of methods}, The American Mathematical Monthly \textbf{128} (2021), no.~8, 701--718.

\bibitem{cheung2011perfect}
Wang-Chi Cheung and Chris Godsil, \emph{Perfect state transfer in cubelike graphs}, Linear Algebra and Its Applications \textbf{435} (2011), no.~10, 2468--2474.

\bibitem{christandl2004perfect}
Matthias Christandl, Nilanjana Datta, Artur Ekert, and Andrew~J Landahl, \emph{Perfect state transfer in quantum spin networks}, Physical review letters \textbf{92} (2004), no.~18, 187902.

\bibitem{chudnovsky2024prime}
Maria Chudnovsky, Michal Cizek, Logan Crew, J{\'a}n Min{\'a}{\v{c}}, Tung~T. Nguyen, Sophie Spirkl, and Nguyen Duy~T\^an, \emph{On prime {C}ayley graphs}, J. Comb. \textbf{17} (2026), no.~2, 223--252. \MR{5030071}

\bibitem{Ehrlich1968}
G.~Ehrlich, \emph{Unit-regular rings}, Portugaliae Mathematica \textbf{27} (1968), 209--212.

\bibitem{godsil2012state}
Chris Godsil, \emph{State transfer on graphs}, Discrete Mathematics \textbf{312} (2012), no.~1, 129--147.

\bibitem{godsil2025integral}
Chris Godsil and Pablo Spiga, \emph{Integral normal {Cayley} graphs}, Journal of Algebraic Combinatorics \textbf{62} (2025), no.~1, 20.

\bibitem{goldsmith1998unit}
Brendan Goldsmith, Simone Pabst, and Audrey Scott, \emph{On unit sum number of rings and modules},  (1998).

\bibitem{MR1041619}
Ralph~P. Grimaldi, \emph{Graphs from rings}, Proceedings of the {T}wentieth {S}outheastern {C}onference on {C}ombinatorics, {G}raph {T}heory, and {C}omputing ({B}oca {R}aton, {FL}, 1989), vol.~71, 1990, pp.~95--103. \MR{1041619}

\bibitem{honold2001characterization}
Thomas Honold, \emph{Characterization of finite {Frobenius} rings}, Archiv der Mathematik \textbf{76} (2001), no.~6, 406--415.

\bibitem{joly1973equations}
Jean-Ren{\'e} Joly, \emph{Equations et varietes algebriques sur un corps fini}, L'Enseignement math{\'e}matique \textbf{19} (1973), 1--117.

\bibitem{kiani2011energy}
Dariush Kiani, Mohsen Molla~Haji Aghaei, Yotsanan Meemark, and Borworn Suntornpoch, \emph{Energy of unitary {Cayley} graphs and gcd-graphs}, Linear algebra and its applications \textbf{435} (2011), no.~6, 1336--1343.

\bibitem{klotz2007some}
Walter Klotz and Torsten Sander, \emph{Some properties of unitary {{Cayley}} graphs}, The Electronic Journal of Combinatorics \textbf{14} (2007), no.~1, R45, 12 pages.

\bibitem{lamprecht1953allgemeine}
Erich Lamprecht, \emph{Allgemeine theorie der {Gau{\ss}schen} {Summen} in endlichen kommutativen {Ringen}}, Mathematische Nachrichten \textbf{9} (1953), no.~3, 149--196.

\bibitem{Lang2002} Serge Lang, {\it Algebra}. Revised third edition. Graduate Texts in Mathematics, 211. Springer-Verlag, New York, 2002. 
\bibitem{invertible_matrices}
N.~J. Lord, \emph{Matrices as sums of invertible matrices}, Mathematics Magazine (1987).


\bibitem{LN1997} Rudolf Lidl, and Harald Niederreiter, {\it Finite fields}, with a foreword by P. M. Cohn, second edition, Encyclopedia of Mathematics and its Applications, 20. Cambridge University Press, Cambridge, 1997.

\bibitem{maimani2010rings}
HR~Maimani, MR~Pournaki, and S~Yassemi, \emph{Rings which are generated by their units: a graph theoretical approach}, Elemente der Mathematik \textbf{65} (2010), no.~1, 17--25.

\bibitem{GR}
C{\'e}sar~Polcino Milies and Sudarshan~K Sehgal, \emph{An introduction to group rings}, vol.~1, Springer Science \& Business Media, 2002.

\bibitem{minavc2024gcd}
J{\'a}n Min{\'a}{\v{c}}, Tung~T Nguyen, and Nguyen~Duy T{\^a}n, \emph{On the gcd-graphs over polynomial rings}, Canadian Journal of Mathematics (2024), 1--28.

\bibitem{perfect-unitary}
J{\'a}n Min{\'a}{\v{c}}, Tung~T. Nguyen, and Nguyen~Duy T{\^a}n, \emph{A complete classification of perfect unitary {Cayley} graphs}, Galois Journal of Algebra \textbf{2} (2026), no.~1, 50--58.

\bibitem{nagell1970type}
Trygve Nagell, \emph{Sur un type particulier d’unit{\'e}s alg{\'e}briques}, Arkiv f{\"o}r Matematik \textbf{8} (1970), no.~2, 163--184.

\bibitem{nguyen2024certain}
Tung~T. Nguyen and Nguyen~Duy T{\^a}n, \emph{On certain properties of the $ p $-unitary {Cayley} graph over a finite ring}, To appear in the Proceedings of the Fields Institute ``Workshop on Galois Cohomology and Massey Products : A conference in honour of J{\'a}n Min{\'a}{\v{c}}’s 71st birthday", arXiv:2403.05635 (2024).

\bibitem{nguyen2024integral}
\bysame, \emph{Integral {Cayley} graphs over a finite symmetric algebra}, Archiv der Mathematik \textbf{124} (2025), 615--623.

\bibitem{nguyen2026supercharacters}
\bysame, \emph{On ${U}$-unitary {Cayley} graphs over finite rings}, arXiv preprint arXiv 2603.21239 (2026).

\bibitem{nguyen2025perfect}
\bysame, \emph{Perfect state transfer on gcd-graphs over a finite {Frobenius} ring}, To appear in International Journal of Algebra and Computation (2026).

\bibitem{nguyen2025supercharacters}
\bysame, \emph{Supercharacters of finite abelian groups and applications to spectra of $U$-unitary {Cayley} graphs}, arXiv preprint arXiv:2508.10348 (2025).

\bibitem{nguyengcd2026}
\bysame, \emph{On gcd-graphs over finite commutative rings}, To appear in Journal of Algebra and Its Applications (2026).

\bibitem{passman1971infinite}
D.S. Passman, \emph{Infinite group rings}, Lecture notes in pure and applied mathematics, M. Dekker, 1971.

\bibitem{podesta2021_finitefield}
Ricardo~A. Podest{\'a} and Denis~E. Videla, \emph{The {Waring}’s problem over finite fields through generalized {Paley} graphs}, Discrete Mathematics \textbf{344} (2021), no.~5, 112324.

\bibitem{podesta2025GP}
\bysame, \emph{Connected components and non-bipartiteness of generalized {P}aley graphs}, Annals of Combinatorics (2025), 1--25.

\bibitem{podesta_new}
\bysame \emph{On $ k $-th unitary Cayley graphs over finite commutative rings: structure and decompositions.} arXiv preprint arXiv:2606.06774 (2026).

\bibitem{Raphael1974}
R.~Raphael, \emph{Rings which are generated by their units}, J. Algebra \textbf{28} (1974), 199--205.

\bibitem{saxena2007parameters}
Nitin Saxena, Simone Severini, and Igor~E Shparlinski, \emph{Parameters of integral circulant graphs and periodic quantum dynamics}, International Journal of Quantum Information \textbf{5} (2007), no.~03, 417--430.

\bibitem{shekarriz2012total}
Mohammad~Hadi Shekarriz, MH~Shirdareh~Haghighi, and Habib Sharif, \emph{On the total graph of a finite commutative ring}, Communications in algebra \textbf{40} (2012), no.~8, 2798--2807.

\bibitem{Skornyakov1964}
L.~A. Skornyakov, \emph{Complemented modular lattices and regular rings}, Oliver \& Boyd, Edinburgh--London, 1964.

\bibitem{so2006integral}
Wasin So, \emph{Integral circulant graphs}, Discrete Mathematics \textbf{306} (2006), no.~1, 153--158.

\bibitem{Huadong}
Huadong Su and Gaohua Tang, \emph{When unit graphs are isomorphic to unitary {C}ayley graphs of rings?}, J. Algebra Appl. \textbf{24} (2025), no.~11, Paper No. 2550267, 9. \MR{4909024}

\bibitem{suntornpoch2016cayley}
Borworn Suntornpoch and Yotsanan Meemark, \emph{{Cayley} graphs over a finite chain ring and gcd-graphs}, Bulletin of the Australian Mathematical Society \textbf{93} (2016), no.~3, 353--363.

\bibitem{pst-gcd-new}
Issaraporn Thongsomnuk and Yotsanan Meemark, \emph{Perfect state transfer in unitary {C}ayley graphs and gcd-graphs}, Linear Multilinear Algebra \textbf{67} (2019), no.~1, 39--50. \MR{3885879}

\bibitem{zelinsky1954every}
Daniel Zelinsky, \emph{Every linear transformation is a sum of nonsingular ones}, Proceedings of the American Mathematical Society \textbf{5} (1954), no.~4, 627--630.

\end{thebibliography}
\end{document}